\documentclass[reqno]{amsart}
\usepackage[utf8]{inputenc}
\usepackage{xcolor}
\usepackage{units}
\usepackage{url}
\usepackage{enumitem}
\usepackage{dsfont}
\usepackage{amstext}
\usepackage{amsthm}
\usepackage{amssymb}
\usepackage{esint}
\PassOptionsToPackage{normalem}{ulem}
\usepackage{ulem}
\usepackage[pdfusetitle,
 bookmarks=true,bookmarksnumbered=false,bookmarksopen=false,
 breaklinks=false,pdfborder={0 0 0},pdfborderstyle={},backref=page,colorlinks=true]
 {hyperref}
\hypersetup{
 pdfborderstyle=,pdfborderstyle={}}

\makeatletter

\providecolor{lyxadded}{rgb}{0,0,1}
\providecolor{lyxdeleted}{rgb}{1,0,0}
\DeclareRobustCommand{\mklyxadded}[1]{\textcolor{lyxadded}\bgroup#1\egroup}
\DeclareRobustCommand{\mklyxdeleted}[1]{\textcolor{lyxdeleted}\bgroup\mklyxsout{#1}\egroup}
\DeclareRobustCommand{\mklyxsout}[1]{\ifx\\#1\else\sout{#1}\fi}
\numberwithin{equation}{section}
\numberwithin{figure}{section}
\theoremstyle{plain}
\newtheorem{thm}{\protect\theoremname}
\theoremstyle{definition}
\newtheorem{defn}[thm]{\protect\definitionname}
\theoremstyle{plain}
\newtheorem{lem}[thm]{\protect\lemmaname}
\theoremstyle{remark}
\newtheorem{rem}[thm]{\protect\remarkname}
\theoremstyle{definition}
\newtheorem{example}[thm]{\protect\examplename}
\theoremstyle{plain}
\newtheorem{cor}[thm]{\protect\corollaryname}

\usepackage{pdfsync}
\usepackage{graphics}
\usepackage{epstopdf}
\usepackage[all]{xy}
\usepackage{amscd}
\usepackage{enumitem}
\usepackage{mathtools}
\usepackage{stmaryrd}
\usepackage[utf8]{inputenc}
\setlist[enumerate]{leftmargin=*,label=(\roman*),align=left}

\makeatletter
\@namedef{subjclassname@2020}{%
  \textup{2020} Mathematics Subject Classification}
\makeatother

\newcommand{\xyR}[1]{ \makeatletter
\xydef@\xymatrixrowsep@{#1} \makeatother} % end of \xyR
\newcommand{\xyC}[1]{ \makeatletter
\xydef@\xymatrixcolsep@{#1} \makeatother} % end of \xyC
\entrymodifiers={++[ ][F]} % to have more space around objects in diagrams

\newcommand{\ra}{\longrightarrow}
\newcommand{\field}[1]{\mathbb{#1}}
\newcommand{\R}{\field{R}} % reals
\newcommand{\N}{\field{N}} % naturals
\newcommand{\CC}{\field{C}} % naturals
\newcommand{\eps}{\varepsilon} % for the sake of brevity only
\renewcommand{\phi}{\varphi}
\newcommand{\diff}[1]{\ifmmode\mathchoice{\hbox{\rm d}#1}  % displaystyle
 {\hbox{\rm d}#1}  % normal 
 {\scalebox{0.75}{$\hbox{\rm d}#1$}}  % scriptstyle 
 {\scalebox{0.35}{$\hbox{\rm d}#1$}}  % scriptscriptstyle
 \fi} % dt,dx,... for integrals
\newcommand{\abs}[2][\empty]{\ifx#1\empty\left|#2\right|%
\else#1\vert #2 #1\vert\fi}% optional arg=size
\newcommand{\Rtil}{\widetilde \R} % real Colombeau generalized number
\newcommand{\Ctil}{\widetilde \CC} % complex Colombeau generalized number
 
\newcommand{\Eball}{B^{{\scriptscriptstyle \text{\rm E}}}} % ordinary Euclidean ball
\newcommand{\frontRise}[2]{\ifmmode\mathchoice{{\vphantom{#1}}^{\scalebox{0.6}{$#2$}}}  % displaystyle
 {{\vphantom{#1}}^{\scalebox{0.56}{$#2$}}}  % normal 
 {{\vphantom{#1}}^{\scalebox{0.47}{$#2$}}}  % scriptstyle 
 {{\vphantom{#1}}^{\scalebox{0.35}{$#2$}}}\fi} % scriptscriptstyle 
\newcommand{\RCreal}[1]{\frontRise{\R}{#1}\Rtil}
\newcommand{\RCrealinfty}[1]{\frontRise{\R}{#1}\overline{\R}}
\newcommand{\RCcomplex}[1]{\frontRise{\CC}{#1}\Ctil}

\newcommand{\RCrealrho}{\RCreal{\rho}}
\newcommand{\RCcomplexrho}{\RCcomplex{\rho}}

\newcommand{\hyperN}[1]{	\frontRise{\N}{#1}\widetilde{\N}}
\newcommand{\hypNr}{\hyperN{\rho}}
\newcommand{\hypNs}{\hyperN{\sigma}}
\newcommand{\frontRiseDown}[3]{\ifmmode\mathchoice{{\vphantom{#1}}^{\scalebox{0.6}{$#2$}}_{\scalebox{0.6}{$#3$}}}  % displaystyle
 {{\vphantom{#1}}^{\scalebox{0.56}{$#2$}}_{\scalebox{0.56}{$#3$}}}  % normal 
 {{\vphantom{#1}}^{\scalebox{0.47}{$#2$}}_{\scalebox{0.47}{$#3$}}}  % scriptstyle 
 {{\vphantom{#1}}^{\scalebox{0.35}{$#2$}}_{\scalebox{0.35}{$#3$}}}\fi} % scriptscriptstyle 
\newcommand{\RCrealud}[2]{\frontRiseDown{\R}{#1}{#2}\Rtil}
\newcommand{\RCcomplexud}[2]{\frontRiseDown{\CC}{#1}{#2}\Ctil}

\newcommand{\hyperlimarg}[3]{\mathchoice{\frontRise{\lim}{\raisebox{-0.05em}{$#1\hspace{-0.67em}$}}\lim_{#3\in \hyperN{#2}\,}}
{\frontRise{\lim}{#1\hspace{-0.25em}}\lim_{#3\in \hyperN{#2}\,}}
{\frontRise{\lim}{#1\hspace{-0.25em}}\lim_{#3\in \hyperN{#2}\,}}
{\frontRise{\lim}{#1\hspace{-0.25em}}\lim_{#3\in \hyperN{#2}\,}}}
\newcommand{\hyperlim}[2]{\hyperlimarg{#1}{#2}{n}}

\newcommand{\hyplimsarg}[2]{\mathchoice{\frontRise{\lim}{\raisebox{-0.05em}{$\hspace{-0.67em}$}}\lim_{#2\in \hyperN{#1}\,}}
{\frontRise{\lim}{\hspace{-0.25em}}\lim_{#2\in \hyperN{#1}\,}}
{\frontRise{\lim}{\hspace{-0.25em}}\lim_{#2\in \hyperN{#1}\,}}
{\frontRise{\lim}{\hspace{-0.25em}}\lim_{#2\in \hyperN{#1}\,}}}

\newcommand{\gsf}{\frontRise{\mathcal{G}}{\rho}\mathcal{GC}^{\infty}}
\newcommand{\gsfk}[1]{\frontRise{\mathcal{G}}{\rho}\mathcal{GC}^{#1}}

\newcommand{\ghf}{\frontRise{\mathcal{G}}{\rho}\mathcal{GH}}

\newcommand{\hypersumarg}[3]{\mathchoice{\frontRise{\sum}{\raisebox{-0.2em}{$#1\hspace{-0.67em}$}}\sum_{#3\in \hyperN{#2}\,}}
{\frontRise{\sum}{#1\hspace{-0.25em}}\sum_{#3\in \hyperN{#2}\,}}
{\frontRise{\sum}{#1\hspace{-0.25em}}\sum_{#3\in \hyperN{#2}\,}}
{\frontRise{\sum}{#1\hspace{-0.25em}}\sum_{#3\in \hyperN{#2}\,}}}

\newcommand{\hypersum}[2]{\hypersumarg{#1}{#2}{n}}

\newcommand{\sbpt}[1]{#1_{\text{\rm s}}}

\newcommand{\Dgsf}{\frontRise{\mathcal{D}}{\rho}\mathcal{GD}}

\newcommand{\setconv}[2]{\frontRiseDown{\text{\rm c}}{#1}{#2}\text{\rm conv}}

\makeatother

\providecommand{\corollaryname}{Corollary}
\providecommand{\definitionname}{Definition}
\providecommand{\examplename}{Example}
\providecommand{\lemmaname}{Lemma}
\providecommand{\remarkname}{Remark}
\providecommand{\theoremname}{Theorem}

\begin{document}
\title[Complex hyper-power series and generalized analytic functions]{Complex hyper-power series and generalized complex analytic functions}
\author{Sekar Nugraheni \and Paolo Giordano}
\address{\textsc{Sekar Nugraheni \newline Faculty of Mathematics, University
of Vienna, Austria \newline Faculty of Mathematics and Natural Sciences,
Universitas Gadjah Mada, Indonesia}}
\email{sekar.nugraheni@ugm.ac.id}
\address{\textsc{Paolo Giordano \newline Faculty of Mathematics, University
of Vienna, Austria}}
\email{paolo.giordano@univie.ac.at}
\begin{abstract}
This paper studies the equivalence between generalized ho\-lo\-mor\-phic
functions (GHF) and complex analytic functions in the framework of
Robinson-Colombeau generalized numbers. In every non-Archimedean ring,
the use of ordinary series is severely restricted by the topological
property that a series converges (in a topology of infinitesimal neighborhoods)
if and only if its general term is infinitesimal. Consequently, classical
Taylor series representations for generalized functions are limited
to infinitesimal neighborhoods. To overcome this drawback, we introduce
and develop the theory of hyperpower series, defined by summation
over the set of hyperfinite natural numbers. We establish the foundational
algebraic and topological properties of hyperpower series, including
their radii of convergence and sets of convergence. Building on this,
we define generalized complex analytic functions and extend several
fundamental theorems of complex analysis to the GHF setting, specifically,
providing generalizations of Goursat’s theorem, Liouville’s theorem,
the identity theorem, and a Paley-Wiener type theorem.
\end{abstract}

\maketitle

\section{\label{sec:Introduction}Introduction}

A major theorem in complex analysis states that all holomorphic functions
are complex analytic functions, and vice versa. GHF theory, see \cite{NuGi24a,NuGi24b},
is likewise linked to hyperseries and to topological properties of
spaces of GHF, providing a framework that may lead to a Cauchy-Kowalevski
theorem for GHF (cf.~\cite{Aragona2005,Tiwari2023} for details about
series in the Colombeau generalized setting and for hyperseries in
the real Robinson-Colombeau generalized setting). However, the appropriate
notion to deal with continuity of this class of generalized functions
is the sharp topology (see \cite{Mukhammadiev2021,Tiwari2022}). In
particular, this topology on the ring $\RCrealrho$ of Robinson-Colombeau
generalized numbers has to deal with balls having infinitesimal radius
$h\in\RCrealrho_{>0}$, and thus $\frac{1}{n}\not\rightarrow0$ if
$n\rightarrow+\infty$, $n\in\N$, because we never have $\frac{1}{n}<h$
if $h$ is infinitesimal. Moreover, if $m\in\N$, $n\in\N_{\leq m}$,
$h\in\RCrealrho_{>0}$ is an infinitesimal number, and $|x_{k+1}-x_{k}|\leq h^{2}$,
$\forall k\in\N$ such that $n\le k\le m$, then we have 
\begin{equation}
|x_{m}-x_{n}|\leq|x_{m}-x_{m-1}|+\dots+|x_{n+1}-x_{n}|\leq(m-n)h^{2}<h,\label{eq:cauchy}
\end{equation}
because $h$ is infinitesimal and $m-n$ is finite. This implies that
$(x_{n})_{n\in\N}\in\RCrealrho^{\N}$ is a Cauchy sequence if and
only if
\[
\lim_{n\to+\infty}\left|x_{n+1}-x_{n}\right|=0.
\]
As a consequence, a series of Robinson-Colombeau generalized numbers
\begin{equation}
\sum_{k=0}^{\infty}a_{k}\text{ converges}\ \iff\ a_{k}\to0\text{ (in the sharp topology)}.\label{eq:convNonArch}
\end{equation}
Moreover, once again from \eqref{eq:cauchy}, it follows that $a_{k}\approx0$
is infinitesimal for all $k\in\N$ sufficiently large.

As a consequence, even if in \cite[Prop. 4.18]{Vernaeve2008}, it
is proved that each complex analytic Colombeau generalized function
can be written as a Taylor series, necessarily this result holds only
in an infinitesimal neighborhood of each point. For example, if the
Taylor series $\sum_{k\in\N}\frac{z^{k}}{k!}$ converges, then necessarily
$a_{k}:=\frac{\left|z^{k}\right|}{k!}\approx0$ for $k\in\N$ sufficiently
large. Therefore, $\left|z\right|=\left(k!a_{k}\right)^{1/k}\approx0$
because $k\in\N$ is finite. Note that the last step would not hold
anymore if $k\in\RCrealrho$ is an infinite number. The impossibility
to maintain this property in a finite neighborhood is a general drawback
of using ordinary series in a (Cauchy complete) non-Archimedean framework.

In this framework, specifically within the non-Archimedean ring $\RCrealrho$,
the hyper-power series $\hypersum{}{\rho}a_{n}(x-c)^{n}$, is defined
as summation over the set of hyperfinite natural numbers $\hypNr\subset\RCrealrho$
with respect to the gauge $\rho=(\rho_{\eps})\ra0$.

In this article, we restrict our focus to a single gauge, in contrast
to the more general definitions of hyper-power series $\hypersum{\rho}{\sigma}a_{n}(x-c)^{n}$
depending on two gauges $\rho$ and $\sigma$, as given in \cite{Tiwari2023,Tiwari2022}.
In spite of this simplification, we are still able to prove that for
GHF, holomorphicity is equivalent to generalized complex analyticity.
Indeed, the key idea is clearly to reduce the necessary computations
to the case of geometric hyperseries, which require only one gauge
(see Section~\ref{sec:gcaf}).

The structure of the paper is as follow. We start with an introduction
to hyper-power series as well as its basic properties such as its
algebraic properties, the definition of radius of convergence, and
the topological properties of the set of convergence. After presenting
the basic definition of generalized complex analytic functions and
its properties, we extend some classical theorems such as Goursat's
theorem, Liouville's theorem, the identity theorem, and a Paley-Wiener
type theorem. Although we developed an expanded mathematical framework
that generalizes classical theory to a broader class of functions,
we frequently have been able to preserve the fundamental intuition
of the classical approach. Hence, often the classical proofs of these
underlying properties can be directly adapted to our framework.

For the general plan of this series of publications devoted to GHF,
see the introduction of \cite{NuGi24a}. We refer to \cite{Giordano2021}
for notions such as the ring of Robinson-Colombeau $\RCcomplexrho$,
the sharp topology, internal sets, and strong internal set, \cite{NuGi24a}
for the notion of basic functions, limits, little-oh, and GHF, and
\cite{NuGi24b} for notions such as generalized $\mathcal{C}^{k}$-functions,
paths, and path integrals as well as their notations and properties.

The work only requires \cite{NuGi24a,NuGi24b} as prior knowledge,
in the sense that it includes all the statements required in the demonstrations
and, of course, the relevant citations.

\section{\label{sec:hyperpowerseries}Hyper-power series}

\subsection{Hypernatural numbers and hypersequences}

We start by defining the set of \emph{hypernatural numbers} in $\RCrealrho$
and the set of $\rho$\emph{-moderate nets of natural numbers}. For
a deeper study of these notions, see \cite{Mukhammadiev2021}.
\begin{defn}
We set:
\begin{enumerate}
\item $\hypNr:=\left\{ [n_{\eps}]\in\RCrealrho:n_{\eps}\in\N\;\forall\eps\right\} $
\item $\N_{\rho}:=\left\{ (n_{\eps})\in\R_{\rho}:n_{\eps}\in\N\;\forall\eps\right\} $.
\end{enumerate}
\end{defn}

Therefore, $n\in\hypNr$ if and only if there exists $(x_{\eps})\in\R_{\rho}$
such that $n=\left[{\rm int}(|x_{\eps}|)\right]$, where ${\rm int}(-)$
is the integer part function. Clearly, $\N\subset\hypNr$. Note that
${\rm int}(-)$ is not well-defined on $\RCrealrho$. In fact, if
$x=1=\left[1-\rho_{\eps}^{\nicefrac{1}{\eps}}\right]=\left[1+\rho_{\eps}^{\nicefrac{1}{\eps}}\right]$,
then ${\rm int}\left(1-\rho_{\eps}^{\nicefrac{1}{\eps}}\right)=0$
whereas ${\rm int}\left(1+\rho_{\eps}^{\nicefrac{1}{\eps}}\right)=1$,
for $\eps$ sufficiently small. Similar counter examples can be set
for floor $\lfloor-\rfloor$ and ceiling $\lceil-\rceil$ functions.
However, the nearest integer function is well defined on $\hypNr$,
as proved in the following:
\begin{lem}
\label{lem:rpi-rni}Let $(n_{\eps})\in\N_{\rho}$ and $(x_{\eps})\in\R_{\rho}$
be such that $[n_{\eps}]=[x_{\eps}]$. Let ${\rm rpi}:\R\rightarrow\N$
be the function rounding to the nearest integer with tie breaking
towards $+\infty$, i.e.~${\rm rpi}(x)=\left\lfloor x+\frac{1}{2}\right\rfloor $.
Then ${\rm rpi}(x_{\eps})=n_{\eps}$ for $\eps$ small. The same result
holds using ${\rm rni}:\R\rightarrow\N$, the function rounding half
towards $-\infty$.
\end{lem}

\begin{proof}
For $\eps$ small, $\rho_{\eps}<\frac{1}{2}$ and since $[n_{\eps}]=[x_{\eps}]$,
for sufficiently small $\eps$, we also have $n_{\eps}-\rho_{\eps}+\frac{1}{2}<x_{\eps}+\frac{1}{2}<n_{\eps}+\rho_{\eps}+\frac{1}{2}$.
But, $n_{\eps}\leq n_{\eps}-\rho_{\eps}+\frac{1}{2}<n_{\eps}+1$.
Therefore, $\left\lfloor x+\frac{1}{2}\right\rfloor =n_{\eps}$. An
analogous argument can be applied to ${\rm rni}(-)$.
\end{proof}
Actually, this lemma does not allow us to define a \textit{nearest
integer} function ${\rm ni}:\hypNr\ra\N_{\rho}$ as ${\rm ni}\left(\left[x_{\eps}\right]\right):={\rm rpi}\left(x_{\eps}\right)$
because if $[x_{\eps}]=[n_{\eps}]$, the equality $n_{\eps}={\rm rpi}\left(x_{\eps}\right)$
holds only for $\eps$ small. A simpler approach is to choose a representative
$(n_{\eps})\in\N_{\rho}$ for each $x\in\hypNr$ and to define ${\rm ni}(x):=\left(n_{\eps}\right)$.
Clearly, we must consider the net $\left({\rm ni}(x)_{\eps}\right)$
only for $\eps$ small, such as in equalities of the form $x=\left[{\rm ni}(x)_{\eps}\right]$.
This is what we do in the following:
\begin{defn}
The nearest integer function $\text{ni}(-)$ is defined by:
\begin{enumerate}
\item ${\rm ni}:\hypNr\rightarrow\N_{\rho}$
\item If $[x_{\eps}]\in\hypNr$ and ${\rm ni}\left(\left[x_{\eps}\right]\right)=\left(n_{\eps}\right)$
then $\forall^{0}\eps:n_{\eps}={\rm rpi}(x_{\eps})$.
\end{enumerate}
\end{defn}

\noindent In other words, if $x\in\hypNr$, then $x=\left[{\rm ni}(x)_{\eps}\right]$
and ${\rm ni}(x)_{\eps}\in\N$ for all $\eps$. Another possibility
is to reformulate Lem.~\ref{lem:rpi-rni} as
\[
[x_{\eps}]\in\hypNr\Longleftrightarrow[x_{\eps}]=\left[{\rm rpi}(x_{\eps})\right].
\]
Therefore, without loss of generality we may always suppose that $x_{\eps}\in\N$
whenever $[x_{\eps}]\in\hypNr$.
\begin{rem}
\label{rem:hypNr}~
\begin{enumerate}
\item $\hypNr$ with the order $\leq$ induced by $\RCrealrho$, is a directed
set.
\item Generally speaking, if $m$, $n\in\hypNr$, $m^{n}\notin\hypNr$ because
net $\left(m_{\eps}^{n_{\eps}}\right)$ can grow faster than any power
$\rho_{\eps}^{-k}$.
\item If $m\in\hypNr$, then $1^{m}:=\left[(1+z_{\eps})^{m_{\eps}}\right]$,
where $(z_{\eps})$ is $\rho$-negligible, is well defined and $1^{m}=1$.
In fact, $\log\left(1+z_{\eps}\right)^{m_{\eps}}$ is asymptotically
equal to $m_{\eps}z_{\eps}\ra0$, and this shows that $\left(1+z_{\eps}\right)^{m_{\eps}}$
is moderate. Finally, $\left|\left(1+z_{\eps}\right)^{m_{\eps}}-1\right|\leq\left|z_{\eps}\right|m_{\eps}\left(1+z_{\eps}\right)^{m_{\eps}-1}$
by the mean value theorem.
\end{enumerate}
\end{rem}

A map $a:\hypNr\ra\RCcomplexrho$ is called $\rho$-\textit{hypersequence
of generalized complex numbers}. The value $a(n)\in\RCcomplexrho$
at $n\in\hypNr$ of the function $a$ are called \textit{terms} of
the hypersequence, and as usual, denoted using an index as argument:
$a_{n}=a(n)$.
\begin{defn}
Let $(a_{n})_{n}:\hypNr\ra\RCcomplexrho$ be a $\rho$-hypersequence
of generalized numbers. Finally let $l\in\RCcomplexrho$. Then we
say that
\[
l\text{ is the hyperlimit of }(a_{n})_{n}
\]
if 
\[
\forall q\in\N\,\exists M\in\hypNr\,\forall n\in\hypNr:n\geq M\Rightarrow|a_{n}-l|<\diff\rho^{q}.
\]
By using the information that the sharp topology on $\RCcomplexrho$
is Hausdorff (see \cite{Giordano2018}), it follows that there exists
at most one hyperlimit. In this case, we will therefore use the notation
\[
l=\hyperlim{}{\rho}a_{n}.
\]
\end{defn}

\begin{example}
\label{exa:hypsequence}~
\begin{enumerate}
\item A sufficient condition to extend an ordinary sequence $a:\N\rightarrow\RCcomplexrho$
of $\rho$-generalized complex numbers to the whole $\hypNr$ is 
\begin{equation}
\forall n\in\hypNr:\left(a_{{\rm ni}(n)_{\eps}}\right)\in\CC_{\rho}.\label{eq:extensionhypseq}
\end{equation}
In fact, in this way $a_{n}$ is well-defined; on the other hand,
using \eqref{eq:extensionhypseq}, we have defined an extension of
the old sequence $a$ because if $n\in\N$, then ${\rm ni}(n)_{\eps}=n$
for $\eps$ small and hence we get $a_{n}=[a_{n}]$. For example,
the sequence of infinities $a_{n}=\frac{i}{n}+\diff\rho^{-1}\in\RCcomplexrho$
for all $n\in\N$ can be extended to $\hypNr$.
\item \label{enu:hypsequencenotconverge}Generally, we restrict our attention
to hypersequences of generalized numbers $(a_{n})_{n}:\hypNr\ra\RCcomplexrho$
that are well-defined for all $n\in\hypNr$ and converge in $\RCcomplexrho$.
However, in \cite{Giordano2021,Mukhammadiev2021}, it is proved that
$\not\exists\hyperlim{}{\rho}\frac{1}{\log n}$ in $\RCrealrho$ (this
hyperlimit converges to zero only using two gauges).
\item For all $k\in\N_{>0}$, we have $\hyperlim{}{\rho}\frac{1}{n^{k}}=0$.
In fact, for all $n\in\hypNr_{>0}$, we have $0<\frac{1}{n^{k}}<\diff\rho^{q}$
if and only if $n^{k}>\diff\rho^{-q}$, i.e. $n>\diff\rho^{-\frac{q}{k}}$.
Thus, it suffices to take $M_{\eps}:={\rm int}\left(\rho_{\eps}^{-\frac{q}{k}}\right)+1$
in the definition of hyperlimit. Analogously, we can treat rational
functions having degree of denominator greater or equal to that of
numerator.
\end{enumerate}
\end{example}

Extending an $\eps$-wise convergent ordinary sequence $(a_{n})_{n\in\N}:\N\ra\RCcomplexrho$
of $\rho$-generalized numbers to a convergent hyperseries is not
an easy task because the gauge $\rho$ must depend on the speed of
convergence of $(a_{n\eps})$ as $n\to+\infty$, $n\in\N$:
\begin{thm}
\label{thm:epslimitofsequence}Let $(a_{n,\eps})_{n,\eps}:\N\times I\ra\CC$.
Assume that for all $\eps$
\[
\exists\lim_{n\rightarrow+\infty}a_{n,\eps}=:l_{\eps},
\]
i.e.
\begin{equation}
\forall\eps\,\forall q\in\N\,\exists M_{\eps q}\in\N_{>0}\,\forall n>M_{\eps q}:|a_{n,\eps}-l_{\eps}|<\rho_{\eps}^{q},\label{eq:epsconvergent}
\end{equation}
such that $l_{\eps}\in\CC$ and (without loss of generality)
\begin{equation}
M_{\eps q}\leq M_{\eps q+1}\quad\forall\eps\forall q.\label{eq:Meps}
\end{equation}
Set $M_{\eps}:=M_{\eps,\rho_{\eps}^{-1}}^{-1}$. Then, if $(M_{\eps})\in\R_{\rho}$:
\begin{enumerate}
\item $M_{q}:=[M_{\eps q}]\in\hypNr$ for all $q\in\N$;
\item \label{enu:negligiblelimit}If $(l_{\eps})$ is $\rho$-moderate,
then there exists $M\in\hypNr$ and a hypersequence $(a_{n})_{n}:\hypNr\ra\RCcomplex{\rho}$
such that $a_{n}=\left[a_{{\rm ni}(n)_{\eps},\eps}\right]\in\RCcomplex{\rho}$
for all $n\in\hypNr_{\geq M}$ and $l:=\left[l_{\eps}\right]=\hyperlim{}{\rho}a_{n}$.
\end{enumerate}
\end{thm}

\begin{proof}
Since $\forall q\in\N$ $\forall^{0}\eps:q\leq\rho_{\eps}^{-1}$,
\eqref{eq:Meps} implies
\begin{equation}
\forall q\in\N\,\forall^{0}\eps:M_{\eps}\geq M_{\eps q}.\label{eq:Mepsbar}
\end{equation}
We can hence define $M:=\left[M_{\eps}\right]\in\hypNr$, $M_{q}:=\left[M_{\eps q}\right]\in\hypNr$
because of the assumption $(M_{\eps})\in\R_{\rho}$ and of \eqref{eq:Mepsbar}.
Set
\begin{equation}
a_{n}:=\begin{cases}
\left[a_{{\rm ni}(n)_{\eps},\eps}\right] & \text{if }n>M_{1}\text{ in }\hypNr\\
1 & \text{otherwise}
\end{cases}\quad\forall n\in\hypNr.\label{eq:a_nsigma}
\end{equation}
We have to prove that this well-defines a hypersequence $(a_{n})_{n}:\hypNr\ra\RCcomplex{\rho}$.
First of all, the sequence is well-defined with respect to the equality
in $\hypNr$. Moreover, setting $q=1$ in \eqref{eq:epsconvergent},
we get $|a_{n,\eps}-l_{\eps}|<\rho_{\eps}$ for all $\eps$ and for
all $n>M_{\eps1}$. If $n>M_{1}$ in $\hypNr$, then ${\rm ni}(n)_{\eps}>M_{\eps1}$
for $\eps$ small, and hence $|a_{{\rm ni}(n)_{\eps},\eps}-l_{\eps}|<\rho_{\eps}$.
This shows that $a_{n}\in\RCcomplex{\rho}$ and $l=[l_{\eps}]\in\RCcomplex{\rho}$
because we assumed that $(l_{\eps})\in\CC_{\rho}$. Finally, \eqref{eq:Meps}
and \eqref{eq:Mepsbar} yield that if $n>M_{q}$ then $n>M_{1}$ and
hence $\left|a_{n}-l\right|<\diff\rho^{q}$.
\end{proof}
In some cases, it could happen that $\rho_{\eps}^{-1}<M_{\eps q}$
for all $q\in\N$ at least on some subpoint $L\subseteq_{0}I$. In
these cases, we have to change gauge, e.g.~taking $\sigma_{\eps}:=M_{\eps,\rho_{\eps}^{-1}}^{-1}$.
For example, we can have $\sigma_{\eps}=\exp\left(-\rho_{\eps}^{-\rho_{\eps}^{-1}}\right)\leq M_{\eps q}^{-1}=\exp\left(-\rho_{\eps}^{-q}\right)\leq\rho_{\eps}=\eps$
and we are hence forced to change gauge choosing the smaller $\sigma<\rho$.
This is actually what happens in Example \ref{exa:hypsequence}\ref{enu:hypsequencenotconverge}.
The problem is that the $\eps$-wise convergence assumption \eqref{eq:epsconvergent}
allows one to obtain only 
\[
\forall q\in\N\,\exists M\in\hypNs\,\forall n\in\hypNs_{\geq M}:\left|a_{n}-\ell\right|<\diff\rho^{q}\quad\text{i.e.}\quad\hyperlim{\rho}{\sigma}a_{n}=\ell.
\]
Therefore, the assumption $(M_{\eps})\in\R_{\rho}$ is a consequence
of our idea to simplify the theory of hyperseries for GHF using only
one gauge. However, note that we have shown this result only for completeness,
since we essentially never use it for the results related to GHF.

The following result applies to all sharply continuous functions (and
hence to all GHF, GSF, and all Colombeau generalized functions, see
e.g.~\cite{NuGi24a,Giordano2021}; see also \cite{Aragona2005} for
a more general class of functions) because of their continuity in
the sharp topology.
\begin{thm}
Suppose that $f:U\subseteq\RCcomplexrho\ra\RCcomplexrho$. Then $f$
is sharply continuous function at $z=c$ if and only if it is hyper-sequentially
continuous, i.e.~for any hypersequence $(z_{n})_{n}$ in $U$ converging
to $c$, the hypersequence $\left(f(z_{n})\right)_{n}$ converges
to $f(c)$, i.e.~$f\left(\hyperlim{}{\rho}z_{n}\right)=\hyperlim{}{\rho}f(z_{n})$.
\end{thm}

\begin{proof}
We only prove that the hyper-sequential continuity is a sufficient
condition, because the other implication is a trivial generalization
of the classical one. By contradiction, assume that for some $Q\in\N$
\begin{equation}
\forall\delta\in\RCrealrho_{>0}\,\exists z_{\delta}\in U:|z_{\delta}-c|<\diff\rho^{n},\;\left|f(z_{\delta})-f(c)\right|>_{s}\diff\rho^{Q},\label{eq:contracontinuous}
\end{equation}
see \cite[Sec.~1.3]{NuGi24a} for the relations $>_{s}$, $=_{s}$,
and more generally, for the language of subpoints. For $n\in\hypNr_{>0}$
set $\delta_{n}:=\frac{1}{n}$ and $z_{n}:=z_{\delta_{n}}$. Then
for all $n\in\hypNr$, from \eqref{eq:contracontinuous} we get $|z_{n}-c|<\diff\rho^{n}\rightarrow0$
as $n\rightarrow+\infty$ in $n\in\hypNr$. Therefore, $(z_{n})_{n}$
is an hypersequence of $U$ that converges to $c$, which yields $f(z_{n})\rightarrow f(c)$,
in contradiction with \eqref{eq:contracontinuous}.
\end{proof}
In the following, we deal with Cauchy hypersequences.
\begin{defn}
\label{def:Cauchyhypsequence}We say $(z_{n})_{n\in\hypNr}$ is a
Cauchy hypersequence if 
\[
\forall q\in\N\,\exists M\in\hypNr\,\forall n,m\in\hypNr_{\geq M}:\left|z_{n}-z_{m}\right|<\diff\rho^{q}.
\]
\end{defn}

\begin{thm}
\label{thm:Cauchycriteria}A hypersequence converges if and only if
it is a Cauchy hypersequence.
\end{thm}

\noindent See \cite{Mukhammadiev2021} for the proof.

In the following results, we generalized the classical algebraic operations
and relations between limits. Thanks to definition of sharp topology,
these results can be proved by trivially generalizing classical proofs.
\begin{thm}
\label{thm:classicaloperations}Let $z$, $w:\hypNr\ra\RCcomplexrho$
be hypersequences, then we have:
\begin{enumerate}
\item If $\hyperlim{}{\rho}z_{n}=a$ and $\hyperlim{}{\rho}w_{n}=b$ then
$\hyperlim{}{\rho}\left(z_{n}+w_{n}\right)=a+b$, $\hyperlim{}{\rho}\left(z_{n}\cdot w_{n}\right)=a\cdot b$.
\item If $w_{n}$ is invertible for all $n\in\hypNr$ then $\hyperlim{}{\rho}\left(\frac{z_{n}}{w_{n}}\right)=\frac{a}{b}$.
\item If for all $n\in\hypNr$, $z_{n}=x_{n}+iy_{n}$, then 
\[
\hyperlim{}{\rho}z_{n}=\hyperlim{}{\rho}x_{n}+i\hyperlim{}{\rho}y_{n},
\]
if these hyperlimits exist.
\end{enumerate}
\end{thm}

\subsection{Definition and properties of hyperseries}

In this section, we introduce the non-Archimedean ring of sequences
of complex hyperseries. Using a classical quotient of suitable moderate
nets $\left(a_{n\eps}\right)_{n,\eps}$ over negligible ones, we can
give a meaning to hyperfinite sums of the form $\sum_{n=N}^{M}a_{n}$
for all $N$, $M\in\hypNr$. This is not a trivial step since, e.g.,
there is a continuum amount of hyperfinite numbers even between $N\in\N$
and $N+1$. Of course, considering the hyperlimit of the hypersequence
$N\in\hypNr\longmapsto\sum_{n=0}^{N}a_{n}\in\RCcomplexrho$, we obtain
the definition of convergent hyperseries. For more details and proofs
about the basic notions introduced here, the reader can refer to \cite{Tiwari2022}
for the real cases.
\begin{defn}
\label{def:hyperseries}The quotient set $\left(\CC^{\N\times I}\right)_{\rho}\slash\sim_{\rho}=:\RCcomplexrho_{{\rm s}}$
of the set $\left(\CC^{\N\times I}\right)_{\rho}$ of nets which are
$\rho$\textit{-moderate over hypersums}, i.e.~satisfying
\begin{equation}
\left(a_{n\eps}\right)_{n,\eps}\in\left(\CC^{\N\times I}\right)_{\rho}\;:\Longleftrightarrow\;\forall N\in\hypNr:\left(\sum_{n=0}^{{\rm ni}(N)_{\eps}}a_{n\eps}\right)\in\CC_{\rho},\label{eq:moderateseries}
\end{equation}
by $\rho$\textit{-negligible nets}, i.e.~such that $\left(a_{n\eps}\right)_{n,\eps}\sim_{\rho}0$,
where:
\begin{equation}
\left(a_{n\eps}\right)_{n,\eps}\sim_{\rho}\left(\bar{a}_{n\eps}\right)_{n,\eps}\;:\Longleftrightarrow\;\forall M,N\in\hypNr:\left(\sum_{n={\rm ni}(N)_{\eps}}^{{\rm ni}(M)_{\eps}}(a_{n\eps}-\bar{a}_{n\eps})\right)\sim_{\rho}0\label{eq:negligibleseries}
\end{equation}
is called the \textit{space of sequences for complex hyperseries}
(this name explains the meaning of the subscript in $\RCcomplexrho_{{\rm s}}$).
Nets of $\CC^{\N\times I}$ are denoted as $(a_{n\eps})_{n,\eps}$
or simply $(a_{n\eps})$; equivalence classes of $\RCcomplexrho_{{\rm s}}$
are denoted as $(a_{n})_{n}=[(a_{n\eps})_{n,\eps}]=[a_{n\eps}]\in\RCcomplexrho_{{\rm s}}$.
Note that in the term $a_{n\eps}$, we have $n\in\N$, and not $n\in\hypNr$.
Similarly, we can define the ring $\RCrealud{\rho}{}{}_{\mathrm{s}}:=\left(\R^{\N\times I}\right)_{\rho}\slash\sim_{\rho}$.\\
For some results, such as e.g.~the ratio test Thm.~\ref{thm:ratiotest}
and the root test Thm.~\ref{thm:roottest} (where we need to perform
term by term operations in a hyperseries) the following \emph{uniform}
notions or moderateness will be useful:
\begin{align*}
(a_{n\eps})_{n,\eps}\in\left(\CC^{\N\times I}\right)_{\mathrm{u},\rho}\quad & :\iff\quad\forall Q\in\N\,\forall^{0}\eps\,\forall n\in\N:\ |a_{n\eps}|\le\rho_{\eps}^{-Q},\\
(a_{n\eps})_{n,\eps}\sim_{\mathrm{u},\rho}(\bar{a}_{n\eps})_{n,\eps}\quad & :\iff\quad\forall q\in\N\,\forall^{0}\eps\,\forall n\in\N:\ |a_{n\eps}-\bar{a}_{n\eps}|\le\rho_{\eps}^{q}.
\end{align*}
The corresponding quotient ring is denoted by $\RCcomplexrho_{\mathrm{u}}:=\left(\CC^{\N\times I}\right)_{\mathrm{u},\rho}/\sim_{\mathrm{u},\rho}$
and the equivalence classes by $[a_{n\eps}]_{\mathrm{u}}\in\RCcomplexrho_{\mathrm{u}}$.
Similarly, we can define the ring $\RCrealud{\rho}{}{}_{\mathrm{u}}:=\left(\R^{\N\times I}\right)_{\mathrm{u},\rho}\slash\sim_{\mathrm{u},\rho}$.
\end{defn}

\noindent Clearly, the adjective \emph{uniform} refers to the logical
structure $\exists\eps_{0}\,\forall\eps\le\eps_{0}\,\forall n\in\N$.
Moreover, it is easy to prove that both $\RCcomplexrho\subseteq\RCcomplexrho_{{\rm s}}$
and $\RCcomplexrho\subseteq\RCcomplexrho_{{\rm u}}$ via the embedding
$[x_{\eps}]\mapsto[(x_{\eps})_{n,\eps}]_{\mathrm{s}(u)}$, and that
the map $\lambda:[a_{n\eps}]_{\mathrm{u}}\in\RCcomplexrho_{\mathrm{u}}\mapsto[a_{n\eps}]_{\mathrm{s}}\in\RCcomplexrho_{\mathrm{s}}$
is well-defined and linear. On the other hand, the possible injectivity
of $\lambda$ is still and open problem.
\begin{defn}
\label{def:convergenthyperseries}Let $(a_{n})_{n}=[a_{n\eps}]\in\RCcomplexrho_{{\rm s}}$,
then the term 
\begin{equation}
\sum_{n=N}^{M}a_{n}:=\left[\sum_{n={\rm ni}(N)_{\eps}}^{{\rm ni}(M)_{\eps}}a_{n\eps}\right]\in\RCcomplexrho\quad\forall N,M\in\hypNr\label{eq:hyperfinitesum}
\end{equation}
is called $\rho$\textit{-hypersum} of $(a_{n})_{n}$. Moreover, we
say that $s$ is the $\rho$\textit{-sum of the hyperseries with terms}
$(a_{n})_{n}$\textit{ as} $n\in\hypNr$ if the hypersequence $N\in\hypNr\longmapsto\sum_{n=0}^{N}a_{n}\in\RCcomplexrho$
converges to $s\in\RCcomplexrho$. In this case, we write 
\begin{equation}
s=\hyperlimarg{}{\rho}{N}\sum_{n=0}^{N}a_{n}=:\hypersum{}{\rho}a_{n}.\label{eq:convergenthypseries}
\end{equation}
As usual, we also say that the hyperseries $\hypersum{}{\rho}a_{n}$
\emph{is convergent} if 
\[
\exists\,\hyperlimarg{}{\rho}{N}\sum_{n=0}^{N}a_{n}\in\RCcomplexrho.
\]
Whereas, we say that a hyperseries $\hypersum{}{\rho}a_{n}$ \emph{does
not converge} if $\hyperlim{}{\rho}\sum_{n=0}^{N}a_{n}$ does not
exist in $\RCcomplexrho$. More specifically, if $\hyperlim{}{\rho}\sum_{n=0}^{N}a_{n}=+\infty\,(-\infty)$,
we say that $\hyperlim{}{\rho}\sum_{n=0}^{N}a_{n}$ \emph{diverges}
to $+\infty$ $(-\infty)$.
\end{defn}

Note that if $N$, $M\in\N$, then $\sum_{n=N}^{M}a_{n}=a_{N}+\dots+a_{M}\in\RCcomplexrho$
is the usual finite sum. In particular, we have the following properties:
\begin{lem}
\label{lem:extensionN}If $N=[N_{\eps}]\in\hypNr$, with $N_{\eps}\in\N$,
and $(a_{n})_{n}=[a_{n\eps}]\in\RCcomplexrho_{{\rm s}}$, then $a_{N}:=\left[a_{N_{\eps},\eps}\right]\in\RCcomplexrho$
is well defined. That is, any sequence $(a_{n})_{n}\in\RCcomplexrho_{{\rm s}}$
can be extended from $\N$ to the entire set $\hypNr$ of hyperfinite
numbers.
\end{lem}

\begin{proof}
In fact, if $(N_{\eps})\in\N_{\rho}$, then from $(a_{n})_{n}\in\RCcomplexrho_{{\rm s}}$
we get the existence of $Q_{i}\in\N$ such that
\[
\left|a_{N_{\eps},\eps}\right|=\left|\sum_{n=0}^{N_{\eps}}a_{n\eps}-\sum_{n=0}^{N_{\eps}-1}a_{n\eps}\right|\leq\rho_{\eps}^{-Q_{1}}+\rho_{\eps}^{-Q_{2}}.
\]
Finally, if $(a_{n})_{n}=\left[a_{n\eps}\right]=\left[\bar{a}_{n\eps}\right]\in\RCcomplexrho_{{\rm s}}$,
then directly from \eqref{eq:negligibleseries} with $M=N$ we get
$\left[\left(a_{{\rm ni}(N)_{\eps},\eps}\right)_{n,\eps}\right]=\left[\left(\bar{a}_{{\rm ni}(N)_{\eps},\eps}\right)_{n,\eps}\right]$.
Therefore, $a_{N}:=[a_{N_{\eps},\eps}]\in\RCcomplexrho$ is well-defined.
In particular, this applies with $N=n\in\N$, so that any equivalence
class $(a_{n})_{n}=\left[\left(a_{n\eps}\right)_{n,\eps}\right]\in\RCcomplexrho_{{\rm s}}$
also defines an ordinary sequence $\left(a_{n}\right)_{n\in\N}=\left(\left[a_{n\eps}\right]\right)_{n\in\N}$
of $\RCcomplexrho$. On the other hand, let us explicitly note that
if $\left(a_{n}\right)_{n\in\N}=\left(\bar{a}_{n}\right)_{n\in\N}$,
i.e. if $a_{n}=\bar{a}_{n}$ for all $n\in\N$, then not necessarily
\eqref{eq:negligibleseries} holds, i.e. we can have $(a_{n})_{n}\neq(\bar{a}_{n})_{n}$
as elements of the quotient module $\RCcomplexrho_{{\rm s}}$.
\end{proof}
\begin{example}
\label{exa:hypseries}The following examples of convergent hyperseries
justify our definition of hyperseries by recovering classical examples
such as geometric and exponential complex hyperseries. We recall that
this is not possible using classical series in a non-Archimedean setting.
\begin{enumerate}
\item Let $N=[N_{\eps}]\in\hypNr$, where $N_{\eps}\in\N$ for all $\eps$,
then $\sum_{n=N}^{N}a_{n}=\left[a_{N_{\eps},\eps}\right]=a_{N}\in\RCcomplexrho$.
We recall that $a_{N}=\left[a_{N_{\eps},\eps}\right]$ is the extension
of $(a_{n})_{n}\in\RCcomplexrho_{{\rm s}}$ (see Lem. \ref{lem:extensionN}).
\item \label{enu:geometryconverge}For all $k\in\RCcomplexrho$, $0<|k|<1$,
we have
\begin{equation}
\hypersum{}{\rho}k^{n}=\frac{1}{1-k}.\label{eq:geometricseries}
\end{equation}
We first note that $|k^{n}|\leq1$ for all $n\in\N$, so that $(k^{n})_{n}\in\RCcomplexrho_{{\rm s}}$.
Now,
\[
\hypersum{}{\rho}k^{n}=\hyperlimarg{}{\rho}{N}\sum_{n=0}^{N}k^{n}=\hyperlimarg{}{\rho}{N}\frac{1-k^{N+1}}{1-k}.
\]
But $|k|^{N+1}<\diff\rho^{q}$ if and only if $(N+1)\log|k|<q\log\diff\rho$.
Since $0<|k|<1$, we have $\log|k|<0$ and we obtain $N>q\frac{\log d}{\log|k|}-1$.
It suffices to take $M_{\eps}:={\rm int}\left(q\frac{\log\rho_{\eps}}{\log|k|_{\eps}}\right)$
in the definition of hyperlimit.
\item \label{enu:geometrydiverge}Let $k\in\RCrealrho_{>0}$ such that $|k|\sbpt{>}1$
then the hyperseries $\hypersum{}{\rho}k^{n}$ is not convergent.
In fact, by contradiction, in the opposite case we would have $\sum_{n=0}^{N}k^{n}\in(l-1,l+1)$
for some $l\in\RCrealrho$ for all $N$ sufficiently large, but this
is impossible because for all fixed $L\in\RCrealrho_{>0}$ and for
$N$ sufficiently large,
\[
\sum_{n=0}^{N}k^{n}=\hyperlimarg{}{\rho}{N}\frac{1-k^{N+1}}{1-k}\sbpt{>}L,
\]
because $\hyperlimarg{}{\rho}{N}k^{N+1}=_{s}+\infty$.
\item \label{enu:expHyperseries}For all $z\in\RCcomplexrho$ such that
$|z|<-K\log\diff\rho$, for some $K\in\R$, we have $\hypersum{}{\rho}\frac{z^{n}}{n!}=e^{z}$.
Since $|z|<-K\log\diff\rho$, we have $|z_{\eps}|\leq-K\log\rho_{\eps}$
for $\eps$ small. Thereby, $\frac{|z_{\eps}|^{n}}{n!}\leq e^{-K\log\rho_{\eps}}=\rho_{\eps}^{-K}$
for all $n\in\N$ and thus $\left(\frac{z^{n}}{n!}\right)_{n}\in\RCcomplexrho_{{\rm s}}$.
For all $N=[N_{\eps}]\in\hypNr$, $N_{\eps}\in\N$, and all $\eps$,
we have
\begin{equation}
\sum_{n=0}^{N_{\eps}}\frac{z_{\eps}^{n}}{n!}=e^{z_{\eps}}-\sum_{n=N_{\eps}+1}^{+\infty}\frac{z_{\eps}^{n}}{n!}.\label{eq:tailofe}
\end{equation}
Now, take $N\in\hypNr$ such that $\frac{-K\log\diff\rho}{N+1}<\frac{1}{2}$,
so that we can assume $N_{\eps}+1>2\left(\log\rho_{\eps}^{-K}\right)$
for all $\eps$. We have $\left|\sum_{n=N_{\eps}+1}^{+\infty}\frac{z_{\eps}^{n}}{n!}\right|\leq\sum_{n>N_{\eps}}\frac{\left(\log\rho_{\eps}^{-K}\right)^{n}}{n!}$,
and for all $n\geq N_{\eps}$, we have (by induction)
\[
\frac{\left(\log\rho_{\eps}^{-K}\right)^{n+1}}{(n+1)!}<\frac{1}{2^{n+1}}.
\]
Therefore $\left|\sum_{n=N_{\eps}+1}^{+\infty}\frac{z_{\eps}^{n}}{n!}\right|\leq\sum_{n>N_{\eps}}\frac{1}{2^{n}}$
and hence $\hyperlimarg{}{\rho}{N}\sum_{n=N+1}^{+\infty}\frac{z^{n}}{n!}=0$
by \eqref{eq:geometricseries}. This and \eqref{eq:tailofe} yield
the conclusion.
\end{enumerate}
\end{example}

The following result allows us to obtain hyperseries by considering
$\eps$-wise convergence of its summands. Its proof is clearly very
similar to that of Thm.~\ref{thm:epslimitofsequence}, but with a
special attention to the condition $(a_{n})_{n}\in\RCcomplexrho_{{\rm s}}$.
\begin{thm}
\label{thm:epslimitofseries}Let $(a_{n,\eps})_{n,\eps}:\N\times I\ra\CC$.
Assume that for all $\eps$, the standard series $\sum_{n=0}^{+\infty}a_{n,\eps}$
converges to $s_{\eps}\in\CC$, i.e.
\begin{equation}
\forall\eps\,\forall q\in\N\,\exists M_{\eps q}\in\N_{>0}\,\forall N\geq M_{\eps q}:\left|\sum_{n=0}^{N}a_{n,\eps}-s_{\eps}\right|<\rho_{\eps}^{q},\label{eq:epsconvergent-1}
\end{equation}
such that $l_{\eps}\in\CC$ and (without loss of generality)
\begin{equation}
M_{\eps q}\leq M_{\eps q+1}\quad\forall\eps\forall q.\label{eq:uniformlyincrease}
\end{equation}
Set $M_{\eps}:=M_{\eps q_{\eps}^{-1}}^{-1}$ . Then, if $\left(M_{\eps}\right)\in\R_{\rho}$:
\begin{enumerate}
\item $M_{q}:=[M_{\eps q}]\in\hypNr$ for all $q\in\N$;
\item If $(s_{\eps})\in\CC_{\rho}$ then there exists $M\in\hypNs$ such
that
\begin{enumerate}
\item $s_{M}:=\left[\sum_{n=0}^{M_{\eps}}a_{n,\eps}\right]\in\RCcomplex{\sigma}$;
\item $(a_{n+M})_{n}:=\left[\left(a_{n+M_{\eps},\eps}\right)_{n,\eps}\right]_{s}\in\RCcomplex{\sigma}_{{\rm s}}$;
\item $s=s_{M}+\hypersum{}{\sigma}{a_{n+M}}$.
\end{enumerate}
\end{enumerate}
\end{thm}

\begin{proof}
Since $\forall q\in\N\,\forall^{0}\eps:q\leq\rho_{\eps}^{-1}$, \eqref{eq:uniformlyincrease}
implies
\begin{equation}
\forall q\in\N\,\forall^{0}\eps:M_{\eps}\geq M_{\eps q}.\label{eq:Meps-1}
\end{equation}
We can hence $M:=[M_{\eps}]\in\hypNr$, $M_{q}:=[M_{\eps q}]\in\hypNr$
because of the assumption $\left(M_{\eps}\right)\in\R_{\rho}$ and
of \eqref{eq:Meps-1}. Since for all $\eps$ condition \eqref{eq:epsconvergent-1}
yields $\left|\sum_{n=0}^{N}a_{n,\eps}-s_{\eps}\right|<\rho_{\eps}^{q}<\sigma_{\eps}^{q}<1$
for all $N\geq M_{\eps q}$ and if $N=M$, this gives $s_{M}=\left[\sum_{n=0}^{M_{\eps}}a_{n,\eps}\right]\in\RCcomplex{\rho}$,
because we assumed that $s:=\left[s_{\eps}\right]\in\RCcomplex{\rho}$.
If $N\in\hypNr$, then $\text{ni}(N)_{\eps}+M_{\eps}\geq M_{\eps}$
for all $\eps$, and hence 
\[
\left|\sum_{n=0}^{M_{\eps}-1}a_{n,\eps}+\sum_{n=M_{\eps}}^{{\rm ni}(N)_{\eps}+M_{\eps}}a_{n,\eps}-s_{\eps}\right|=\left|s_{M_{\eps}}+\sum_{n=0}^{{\rm ni}(N)_{\eps}}a_{n+M_{\eps},\eps}-s_{\eps}\right|<\rho_{\eps}^{q}<\sigma_{\eps}^{q}<1.
\]
This shows that $\left[\sum_{n=0}^{{\rm ni}(N)_{\eps}}a_{n+M_{\eps},\eps}\right]\in\RCcomplex{\rho}$,
because $s_{M}$, $s\in\RCcomplex{\rho}$, i.e.~$(a_{n+M})_{n}\in\RCcomplex{\rho}_{s}$.
Similarly, proceeding as above, we can prove that $\left|\sum_{n=0}^{N}a_{n+M}-s-s_{M}\right|<\diff\rho^{q}$.
\end{proof}
Once again, in some cases, it could happen that $\rho_{\eps}^{-1}<M_{\eps q}$
for all $q\in\N$ at least on some subpoint $L\subseteq_{0}I$. In
these cases, we have to consider the suitable gauge $\sigma$, e.g.~taking
$\sigma_{\eps}:=\exp\left(-\rho_{\eps}^{-\rho_{\eps}^{-1}}\right)\leq M_{\eps q}^{-1}$.

We now study some basic properties of hyperfinite sums \eqref{eq:hyperfinitesum}
and hyperseries.
\begin{lem}
\label{lem:hyperfinitesum}Let $(a_{n})_{n}\in\RCcomplexud{\rho}{}_{s}$,
and $M$, $N\in\hypNr$, then 
\begin{equation}
\sum_{n=0}^{N+M}a_{n}-\sum_{n=0}^{N}a_{n}=\sum_{n=N+1}^{N+M}a_{n}.\label{eq:hyperfinitelemma}
\end{equation}
\end{lem}

\begin{proof}
For simplicity, if $N=[N_{\eps}]$, $M=[M_{\eps}]$ with $N_{\eps}$,
$M_{\eps}\in\N$ for all $\eps$, then $\sum_{n=0}^{N}a_{n}=\left[\sum_{n=0}^{N_{\eps}}a_{n\eps}\right]\in\RCcomplexrho$
and
\[
\sum_{n=0}^{N+M}a_{n}=\left[\sum_{n=0}^{N_{\eps}+M_{\eps}}a_{n\eps}\right]=\left[\sum_{n=0}^{N_{\eps}}a_{n\eps}+\sum_{n=N_{\eps}+1}^{N_{\eps}+M_{\eps}}a_{n\eps}\right]=\sum_{n=0}^{N}a_{n}+\sum_{n=N+1}^{N+M}a_{n}.
\]
\end{proof}
\begin{lem}
\label{lem:limitseq}If $\hypersum{}{\rho}a_{n}$ is convergent and
$M\in\hypNr$, then 
\[
\hypersum{}{\rho}a_{n}=\hyperlimarg{}{\rho}{N}\sum_{n=0}^{N+M}a_{n}.
\]
Therefore, from \eqref{eq:hyperfinitelemma}, we also have
\[
\hyperlimarg{}{\rho}{N}\sum_{n=N}^{N+M}a_{n}=0.
\]
In particular, $\hyperlim{}{\rho}a_{n}=0$.
\end{lem}

\begin{proof}
Let $s=\hypersum{}{\rho}a_{n}$. Directly from the definition of convergent
hyperseries \eqref{eq:convergenthypseries}, we have
\[
\forall q\in\N\,\exists K\in\hypNr\,\forall N\in\hypNr_{\geq K}:\left|\sum_{n=0}^{N}a_{n}-s\right|<\diff\rho^{q}.
\]
In particular, if $N\in\hypNr_{\geq K}$, then also $N+M\in\hypNr_{\geq K}$,
which is our conclusion.
\end{proof}
Directly from Lem.~\ref{lem:hyperfinitesum}, we also have:
\begin{cor}
\label{cor:tailofhyperseries}Let $\hypersum{}{\rho}a_{n}$ be a convergent
hyperseries. Then adding a hyperfinite number of terms have no effect
on the convergence of the hyperseries, that is
\[
\sum_{n=0}^{K}a_{n}+\hyperlimarg{}{\rho}{N}\sum_{n=K+1}^{N}a_{n}=\hypersum{}{\rho}a_{n}\quad\forall K\in\hypNr.
\]
\end{cor}

Moreover, below we also have the usual connection between convergence
and absolute convergence of hyperseries:
\begin{thm}
\label{thm:absoluteconverge}If the hyperseries $\hypersum{}{\rho}|a_{n}|$
converges then $\hypersum{}{\rho}a_{n}$ converges. Moreover,
\[
\left|\,\hypersum{}{\rho}a_{n}\right|\leq\hypersum{}{\rho}\left|a_{n}\right|.
\]
\end{thm}

\begin{proof}
If we consider the hyperfinite sums $S_{N}=\sum_{n=0}^{N}a_{n}$ and
$T_{N}=\sum_{n=0}^{N}\left|a_{n}\right|$ for all $N\in\hypNr$, then
$(T_{N})_{N}\in\RCcomplexrho_{{\rm s}}$ is a Cauchy hypersequence
by Thm.~\ref{thm:Cauchycriteria}. Then by the triangle inequality
\[
\left|S_{N}-S_{M}\right|=\left|\sum_{n=N+1}^{M}a_{n}\right|\leq\sum_{n=N+1}^{M}|a_{n}|=|T_{M}-T_{N}|.
\]
Thus, $(S_{N})_{N}\in\RCcomplexrho_{{\rm s}}$ is also a Cauchy hypersequence,
so by Thm.~\ref{thm:Cauchycriteria}, it converges.
\end{proof}
In the classical direct comparison test, we need to assume a relation
of the form $a_{n}\leq b_{n}$ for all $n\in\N$ between general terms
of two real series. In case of hyperseries, we define a weaker but
natural relation in $\RCrealud{\rho}{}{}_{\mathrm{s}}$:
\begin{defn}
\label{def:orderhyperseries}Let $(a_{n})_{n}$, $(b_{n})_{n}\in\RCrealud{\rho}{}{}_{\mathrm{s}}$,
then we say that $(a_{n})_{n}\leq(b_{n})_{n}$ if 
\begin{equation}
\forall N,M\in\hypNr:\sum_{n=N}^{M}a_{n}\leq\sum_{n=N}^{M}b_{n}\;\text{in}\;\RCrealrho.\label{eq:orderhyperseries}
\end{equation}
Note explicitly, that if $M<_{L}N$ on $L\subseteq_{0}I$, then $\sum_{n=N}^{M}a_{n}=_{L}\sum_{n=N}^{M}b_{n}=_{L}0$.
Therefore, \eqref{eq:orderhyperseries} is equivalent to 
\[
\forall N,M\in\hypNr:M\geq N\implies\sum_{n=N}^{M}a_{n}\leq\sum_{n=N}^{M}b_{n}\;\text{in}\;\RCrealrho.
\]
Note that the order relation in $\RCrealrho_{\mathrm{u}}$ is $[a_{n\eps}]_{\mathrm{u}}\le[b_{n\eps}]_{\mathrm{u}}$
if for some $[z_{n\eps}]_{\mathrm{u}}=0$, we have
\[
\forall^{0}\eps\,\forall n\in\N:\ a_{n\eps}\le b_{n\eps}+z_{n\eps}.
\]
We will implicitly use this relation in the ratio Thm.~\ref{thm:ratiotest}
and root Thm.~\ref{thm:roottest}.
\end{defn}

\noindent We have that $(\RCrealud{\rho}{}{}_{s},\leq)$ is an ordered
$\RCrealrho$-module, and if $(a_{n})_{n}\geq0$, then $N\in\hypNr\mapsto\sum_{n=0}^{N}a_{n}\in\RCrealrho$
is increasing.

The direct comparison test for hyperseries can now be stated as follows:
\begin{thm}
\label{thm:directcomparisonhyperseries}Let $\hypersum{}{\rho}a_{n}$
and $\hypersum{}{\rho}b_{n}$ be hyperseries of generalized real numbers
with $(a_{n})_{n}$, $(b_{n})_{n}\geq0$ and such that
\begin{equation}
\exists N\in\N:(a_{n+N})_{n}\leq(b_{n+N})_{n}.\label{eq:directcomparison}
\end{equation}
Then we have:
\begin{enumerate}
\item \label{enu:comparisonconverges}If $\hypersum{}{\rho}b_{n}$ is convergent
then so is $\hypersum{}{\rho}a_{n}$.
\item \label{enu:comparisondiverges}If $\hypersum{}{\rho}a_{n}$ is divergent
to $+\infty$, then so is $\hypersum{}{\rho}b_{n}$.
\end{enumerate}
\end{thm}

\begin{proof}
We first note that \eqref{eq:directcomparison} can be simply written
as $(\alpha_{n})_{n}\leq(\beta_{n})_{n}$, where $\alpha_{n}:=a_{n+N}$
and $\beta_{n}:=b_{n+N}$. Thereby, Corollary \ref{cor:tailofhyperseries}
implies that, without loss of generality, we can assume $N=0$. To
prove \ref{enu:comparisonconverges}, let us consider the partial
sum $A_{n}:=\sum_{i=0}^{n}a_{i}$ and $B_{n}:=\sum_{i=0}^{n}b_{i}$,
$n\in\hypNr$. Since $\hypersum{}{\rho}b_{n}$ is convergent, we have
$\exists\hyperlim{}{\rho}B_{n}=:B\in\RCrealrho$. The assumption $(a_{n})_{n}\leq(b_{n})_{n}$
implies $A_{n}\leq B_{n}$. The hypersequences $(A_{n})_{n\in\hypNr}$,
$(B_{n})_{n\in\hypNr}$ are increasing because $(a_{n})_{n}$, $(b_{n})_{n}\geq0$
and hence $(B-B_{n})_{n\in\hypNr}$ decreases to zero because of the
convergence assumption. Now, for all $N$, $M\in\hypNr$, $M\geq N$,
we have
\begin{equation}
A_{N}\leq A_{M}=\sum_{n=0}^{M}a_{n}=\sum_{n=0}^{N}a_{n}+\sum_{n=N+1}^{M}a_{n}\leq A_{N}+\sum_{n=N+1}^{M}b_{n}=A_{N}+(B-B_{N}).\label{eq:comparison}
\end{equation}
Thereby, given $m$, $n\geq N$, applying \eqref{eq:comparison} with
$m$, $n$, we get that both $A_{n}$, $A_{m}$ belongs to the interval
$\left[A_{N},A_{N}+(B-B_{N})\right]$, whose length $B-B_{N}$ decreases
to zero as $N\in\hypNr$ goes to infinity. This shows that $(A_{n})_{n\in\hypNr}$
is a Cauchy hypersequence, and therefore $\hypersum{}{\rho}a_{n}$
converges.

The proof of \ref{enu:comparisondiverges} follows directly from the
inequality $A_{n}\leq B_{n}$ for each $n\in\hypNr$.
\end{proof}
Similar to the classical method, the geometric hyperseries leads to
a useful test for convergence of general hyperseries.
\begin{thm}[Ratio test]
\label{thm:ratiotest}Let $(a_{n})_{n}=[a_{n\eps}]\in\RCcomplexrho_{{\rm s}}$
be such that $|a_{n}|>0$ for all $n\in\hypNr$. Assume that for some
$[k_{\eps}]\in\hypNr$, with $k_{\eps}\in\N$, we have 
\begin{equation}
\exists[L_{\eps}]\in\RCrealrho_{>0}\:\forall^{0}\eps\,\forall n\in\N_{\ge k_{\eps}}:\;\frac{\left|a_{n+1,\eps}\right|}{\left|a_{n\eps}\right|}\leq L_{\eps}.\label{eq:ratio}
\end{equation}
Then, if $L<1$, the hyperseries $\hypersum{}{\rho}a_{n}$ converges
absolutely.
\end{thm}

\begin{proof}
Since $L<1$, by Example \ref{exa:hypseries}\ref{enu:geometryconverge},
the hyperseries $\hypersum{}{\rho}L^{n}$ is convergent. The ratio
assumption \eqref{eq:ratio} yields the existence of $\left(z_{\eps}\right)\sim_{\rho}0$
such that for $\eps$ small (without loss of generality, we can assume
that all the following relations, starting from \eqref{eq:ratio},
hold for all $\eps\le\eps_{0}$), we have $\frac{\left|a_{k_{\eps}+1,\eps}\right|}{\left|a_{k_{\eps},\eps}\right|}\leq L_{\eps}+z_{\eps}$
and hence $\left|a_{k_{\eps}+1,\eps}\right|\leq\left(L_{\eps}+z_{\eps}\right)\left|a_{k_{\eps},\eps}\right|$
and recursively, we have $\left|a_{k_{\eps}+N+1,\eps}\right|\leq\left(L_{\eps}+z_{\eps}\right)\left|a_{k_{\eps}+N,\eps}\right|\le\ldots\le\left(L_{\eps}+z_{\eps}\right)^{N+1}\left|a_{k_{\eps},\eps}\right|$
for $\eps\le\eps_{0}$. Now the series 
\[
\sum_{n=k_{\eps}+1}^{\infty}\left|a_{n\eps}\right|\leq\sum_{n=1}^{\infty}\left(L_{\eps}+z_{\eps}\right)^{n}\left|a_{k_{\eps},\eps}\right|.
\]
Since $\left(z_{\eps}\right)\sim_{\rho}0$ and $L<1$ then $L=\left[L_{\eps}+z_{\eps}\right]<1$,
and hence for all $N=[N_{\eps}]$, $M=[M_{\eps}]\in\hypNr$ with $N_{\eps}\le M_{\eps}$
\[
\sum_{n=\max(N,k)}^{M}\left|a_{n}\right|\leq\left|a_{k}\right|\sum_{n=\max(N,k)}^{M}\left[L_{\eps}+z_{\eps}\right]^{n}=\left|a_{k}\right|\sum_{n=\max(N,k)}^{M}L^{n}.
\]
By direct comparison test for hyperseries of generalized real numbers
(Thm.~\ref{thm:directcomparisonhyperseries}), we can conclude that
$\hypersum{}{\rho}\left|a_{n}\right|$ is convergent (and by Thm.~\ref{thm:absoluteconverge},
the hyperseries $\hypersum{}{\rho}a_{n}$ is also convergent).
\end{proof}
Proceeding as in the previous proof, i.e.~by generalizing the classical
proof for series of complex numbers, we also have the following root
test.
\begin{thm}[Root test]
\label{thm:roottest}Let $(a_{n})_{n}=[a_{n\eps}]\in\RCcomplexrho_{{\rm s}}$.
Assume that for some $[k_{\eps}]\in\hypNr$, with $k_{\eps}\in\N$,
we have 
\[
\exists[L_{\eps}]\in\RCrealrho_{>0}\,\forall^{0}\eps\,\forall n\in\hypNr_{\geq k_{\eps}}:\;\left|a_{n\eps}\right|^{\nicefrac{1}{n}}\leq L_{\eps}.
\]
Then, if $L<1$, the hyperseries $\hypersum{}{\rho}a_{n}$ converges
absolutely.
\end{thm}

\begin{proof}
We can proceed as in the previous theorem considering that if $\left|a_{n}\right|\leq L^{n}$,
then $\sum_{n\in\hypNr_{\geq k}}\left|a_{n}\right|\leq\sum_{n\in\hypNr_{\geq k}}L^{n}$.
\end{proof}

\subsection{Definition and properties of hyper-power series}

Classically, a power series is ``a series of the form $\sum_{n=0}^{+\infty}a_{n}(z-c)^{n}$''
(here, it is clearly meant that $a_{n}$, $z$, $c\in\CC$). If we
do not consider its convergence, we talk of \emph{formal power series}
(then $z$ is some kind of \emph{indeterminate}, not a numerical variable).
Otherwise, we need to take $z$ in a suitable \emph{set of convergence},
and hence the notion of \emph{radius of convergence} is a natural
preliminary step.

In the ring $\RCcomplexrho$, we proceed similarly, but necessarily
being somehow more formal. Firstly, we always have to recall that
our definitions must be independent from representatives. Secondly,
in our setting we can have that the radius of converge $r=+\infty$,
but the values of the hyperseries must be anyway given by moderate
nets. This implies that a given hyperseries (e.g.~the exponential
one) cannot have $B_{r}(0)=\RCcomplexrho$ as \emph{set of convergence}.
\begin{defn}
\label{def:formalHPS}Let $z$, $c\in\RCcomplexrho$. We say $(b_{n})_{n}\in\RCcomplexud{\rho}{}\left\llbracket z-c\right\rrbracket $
if and only if there exist $(a_{n\eps})_{n,\eps}\in\R^{\N\times I}$
and representatives $[z_{\eps}]=z$, $[c_{\eps}]=c$ such that 
\begin{equation}
(b_{n})_{n}=[a_{n\eps}\cdot(z_{\eps}-c_{\eps})^{n}]_{{\rm s}}\in\RCcomplexud{\rho}{}_{{\rm s}}.\label{eq:formalHPS}
\end{equation}
Elements of $\RCcomplexud{\rho}{}\left\llbracket z-c\right\rrbracket $
are called \textit{formal hyperpower series (formal HPS)} because
we are not considering their convergence. In other words, a formal
HPS is a hyperseries (i.e.~an equivalence class $(b_{n})_{n}\in\RCcomplexud{\rho}{}_{{\rm s}}$
in the space of sequences of hyperseries, see Def.~\ref{def:hyperseries},
where we can always consider hyperfinite sums) of the form $\left[a_{n\eps}\cdot(z_{\eps}-c_{\eps})^{n}\right]_{{\rm s}}$.
\end{defn}

\medskip{}

\begin{rem}
~
\begin{enumerate}
\item We explicitly note that $z-c$ is not an indeterminate, like in the
case of formal power series $\CC[z]$, but a generalized number of
$\RCcomplexrho$. For example, if $z-c=y-d$ then $\RCcomplexud{\rho}{}\left\llbracket z-c\right\rrbracket =\RCcomplexud{\rho}{}\left\llbracket y-d\right\rrbracket $.
\item In Example \ref{exa:hypseries}\ref{enu:expHyperseries}, we already
proved that if $z$ satisfies $|z|\leq-K\log\diff\rho$, for some
$K\in\R$, then $\left(\frac{z^{n}}{n!}\right)_{n}=\left[\frac{z_{\eps}^{n}}{n!}\right]_{{\rm s}}\in\RCcomplexud{\rho}{}\left\llbracket z\right\rrbracket $
is a formal HPS. This clearly shows that we can consider any hyperfinite
sum $\sum_{n=N}^{M}\frac{z^{n}}{n!}\in\RCcomplexrho$, $N$, $M\in\hypNr$,
only if $z\in\RCcomplexrho$ is of logarithmic type, so that we can
apply Def.~\ref{def:hyperseries} of hyperseries.
\item Actually, the term $[a_{n\eps}\cdot(z_{\eps}-c_{\eps})^{n}]_{{\rm s}}$
in \eqref{eq:formalHPS} does not depend on $n\in\N$. Similarly,
the notation $(b_{n})_{n}$ does not depend on $n$ and has only the
main aim to recall the sequence of coefficients of the formal HPS.
We have a formally similar situation in the notation $\int f(x)\,\diff x$
that does not depend on $x$.
\end{enumerate}
\end{rem}

As we mentioned above, the previous Def.~\ref{def:formalHPS} sets
immediate problems concerning independence of representatives. Taking
different representatives $[\overline{a}_{n\eps}]=[a_{n\eps}]$, $[\overline{z}_{\eps}]=z$,
$[\overline{c}_{\eps}]=c$, do we have $[a_{n\eps}\cdot(z_{\eps}-c_{\eps})^{n}]_{{\rm s}}=[\overline{a}_{n\eps}\cdot(\overline{z}_{\eps}-c_{\eps})^{n}]_{{\rm s}}$?
Let $a_{n}=[a_{n\eps}]$ be such that $a_{n\eps}:=0$ if $\eps<\frac{1}{n+1}$
and $a_{n\eps}:=1$ otherwise, then $[a_{n\eps}]=0$ but the corresponding
series is not zero:
\[
\forall N=[N_{\eps}]\in\hypNr:\sum_{n=0}^{N_{\eps}}a_{n\eps}=\sum_{n=\left\lceil \frac{1}{\eps}\right\rceil -1}^{N_{\eps}}1=N_{\eps}-\left\lceil \frac{1}{\eps}\right\rceil +2.
\]
This means that $(b_{n})_{n}:=\left[a_{n\eps}(z_{\eps}-c_{\eps})^{n}\right]_{{\rm s}}$
and $(\bar{b}_{n})_{n}:=\left[\bar{a}_{n\eps}(z_{\eps}-c_{\eps})^{n}\right]_{{\rm s}}$,
for $z=1$, $c=0$, yield two different formal HPS and hence, in general
the operation 
\[
((a_{n\eps})_{n,\eps},(z_{\eps}),(c_{\eps}))\in\RCcomplexrho^{\N}\times\RCcomplexrho^{2}\mapsto(b_{n})_{n}:=\left[a_{n\eps}(z_{\eps}-c_{\eps})^{n}\right]_{s}\in\RCcomplexud{\rho}{}_{\mathrm{s}}
\]
is not well-defined. The following notion of negligibility (see Def.~\ref{def:radiusofconvergence}
below) naturally emerges in proving this independence of representative
and in showing that the following definition of radius of convergence
is well-defined (see Lem. \ref{lem:radconvergence}).

The idea to define the radius of convergence corresponding to coefficients
$(a_{n\eps})_{n,\eps}\in\R^{\N\times I}$ is that it does not matter
if 
\[
\left(\limsup_{n\rightarrow+\infty}|a_{n\eps}|^{\nicefrac{1}{n}}\right)^{-1}\in\R\cup\{+\infty\}
\]
yields a non $\rho$-moderate net (for example for $\eps\in L\subseteq_{0}I$)
because this case would intuitively identify a radius of convergence
larger than any number in $\RCrealrho$:
\begin{defn}
\label{def:radiusofconvergence}~
\begin{enumerate}
\item Let $\bar{\R}:=\R\cup\{-\infty,+\infty\}$ be the extended real number
system with the usual (partially defined) operations. We set $\RCrealinfty{\rho}:=\bar{\R}^{I}/\sim_{\rho}$,
where for \emph{arbitrary} $(x_{\eps})$, $(y_{\eps})\in\bar{\R}^{I}$,
as usual, we define
\[
(x_{\eps})\sim_{\rho}(y_{\eps})\quad:\Longleftrightarrow\quad\forall q\in\N\,\forall^{0}\eps:\ |x_{\eps}-y_{\eps}|\leq\rho_{\eps}^{q}.
\]
In $\RCrealinfty{\rho}$, we can also consider the standard order
relation
\[
x\le y\quad:\Longleftrightarrow\quad\exists[x_{\eps}]=x,\:[y_{\eps}]=y\:\forall^{0}\eps:\ x_{\eps}\leq y_{\eps}.
\]
Note that $\left(\RCrealinfty{\rho}\setminus\{-\infty\},+,\leq\right)$
is an ordered group but, since we are considering arbitrary nets $\bar{\R}^{I}$,
the set $\RCrealinfty{\rho}$ is not a ring: e.g.~$+\infty\cdot0$
is still undefined and $+\infty\cdot[z_{\eps}]=[+\infty]$ for all
$(z_{\eps})\in\R_{>0}^{I}$.
\item Moreover, we denote by $\RCcomplexrho_{{\rm c}}:=\left(\CC^{\N\times I}\right)_{\rho}/\simeq_{\rho}$
the quotient ring of \emph{coefficients for HPS}, where 
\begin{equation}
(a_{n\eps})_{n,\eps}\in\left(\CC^{\N\times I}\right)_{\rho}\quad:\Longleftrightarrow\quad\exists Q,R\in\N\,\forall^{0}\eps\,\forall n\in\N:\ |a_{n\eps}|\leq\rho_{\eps}^{-nQ-R}\label{eq:coeffmoderate}
\end{equation}
is the ring of \emph{weakly }$\rho$\emph{-moderate} nets, and
\begin{equation}
(a_{n\eps})_{n,\eps}\simeq_{\rho}(\bar{a}_{n,\eps})_{n,\eps}\quad:\Longleftrightarrow\quad\forall q,r\in\N\,\forall^{0}\eps\,\forall n\in\N:\ |a_{n\eps}-\bar{a}_{n\eps}|\leq\rho_{\eps}^{nq+r},\label{eq:coeffequiv}
\end{equation}
in this case, we say that these two nets are \emph{strongly }$\rho$\emph{-equivalent}.
Equivalence classes of $\RCcomplexrho_{{\rm c}}$ are denoted by $(a_{n})_{{\rm c}}:=[a_{n\eps}]_{{\rm c}}\in\RCcomplexrho_{{\rm c}}$.
\item Finally, if $(a_{n})_{{\rm c}}=[a_{n\eps}]_{{\rm c}}\in\RCcomplexrho_{{\rm c}}$,
then we set $\text{rad}(a_{n})_{n\eps}^{{\rm c}}:=r_{\eps}$ and $\text{rad}(a_{n})_{{\rm c}}=:[r_{\eps}]\in\RCrealinfty{\rho}$,
where
\begin{equation}
r_{\eps}:=\left(\limsup_{n\rightarrow+\infty}|a_{n\eps}|^{\nicefrac{1}{n}}\right)^{-1}\in\R\cup\{+\infty\}.\label{eq:radeps}
\end{equation}
\end{enumerate}
\end{defn}

In the following lemma, we prove that $\text{rad}(a_{n})_{{\rm c}}$
is well-defined:
\begin{lem}
\label{lem:radconvergence}Let $(a_{n})_{{\rm c}}=[a_{n\eps}]_{{\rm c}}=[\bar{a}_{n\eps}]_{{\rm c}}\in\RCcomplexrho_{{\rm c}}$.
Define $r_{\eps}$ as in \eqref{eq:radeps} and similarly define $\bar{r}_{\eps}$
using $\bar{a}_{n\eps}$. Then $(r_{\eps})\sim_{\rho}(\bar{r}_{\eps})$,
and hence $[r_{\eps}]=[\bar{r}_{\eps}]$ in $\RCrealinfty{\rho}$.
\end{lem}

\begin{proof}
For all $\eps\in I$ and all $n\in\N_{>0}$, we have $|\bar{a}_{n\eps}|^{\nicefrac{1}{n}}\leq\left(\left|\bar{a}_{n\eps}-a_{n\eps}\right|+|a_{n\eps}|\right)^{\nicefrac{1}{n}}$.
The binomial formula yields $(x+y)\leq\left(x^{\nicefrac{1}{n}}+y^{\nicefrac{1}{n}}\right)^{n}$
for all $x$, $y\in\R_{\geq0}$, so that $|\bar{a}_{n\eps}|^{\nicefrac{1}{n}}\leq\left|\bar{a}_{n\eps}-a_{n\eps}\right|^{\nicefrac{1}{n}}+|a_{n\eps}|^{\nicefrac{1}{n}}$.
Setting $r=0$ in \eqref{eq:coeffequiv}, for all $q\in\N$ and for
$\eps$ small we have 
\[
\forall n\in\N:\left|a_{n\eps}-\bar{a}_{n\eps}\right|\leq\rho_{\eps}^{nq}.
\]
Therefore, for the same $\eps$ we get $|\bar{a}_{n\eps}|^{\nicefrac{1}{n}}\leq q_{\eps}^{q}+|a_{n\eps}|^{\nicefrac{1}{n}}$.
Taking the limit superior, we obtain $\limsup_{n\rightarrow+\infty}|\bar{a}_{n\eps}|^{\nicefrac{1}{n}}\leq q_{\eps}^{q}+\limsup_{n\rightarrow+\infty}|a_{n\eps}|^{\nicefrac{1}{n}}$.
Inverting the role of $(a_{n\eps})_{n,\eps}$ and $(\bar{a}_{n\eps})_{n,\eps}$
we finally obtain
\[
\forall^{0}\eps:-\rho_{\eps}^{q}\leq\limsup_{n\rightarrow+\infty}|a_{n\eps}|^{\nicefrac{1}{n}}-\limsup_{n\rightarrow+\infty}|\bar{a}_{n\eps}|^{\nicefrac{1}{n}}\leq\rho_{\eps}^{q},
\]
which proves the claim.
\end{proof}
\begin{rem}
~
\begin{enumerate}
\item If $(a_{n})_{{\rm c}}=[a_{n\eps}]_{{\rm c}}\in\RCcomplexrho_{{\rm c}}$,
then for each fixed $n\in\N$, we have that $[(a_{n\eps})_{\eps}]\in\RCcomplexrho$,
i.e. the net $(a_{n\eps})_{\eps}$ is $\rho$-moderate. This is the
main motivation to consider the exponent ``$-R$'' in \eqref{eq:coeffmoderate}
(recall that in our notation $0\in\N$): without the term ``$-R$'',
the only possibility to have $(a_{n})_{{\rm c}}\in\RCcomplexrho_{{\rm c}}$
is that $|a_{0}|\leq1$, which is an unnecessary limitation. Similarly,
we can motivate why we are considering the quantifier ``$\forall n\in\N$''
in the same formula (instead of, e.g., ``$\exists N\in\N\:\forall n\in\N_{\geq N}$'').
Moreover, Thm.~\ref{thm:independencerepresentatives} will motivate
why in \eqref{eq:coeffmoderate} we consider the uniform property
``$\forall^{0}\eps\,\forall n\in\N$'' and not ``$\forall n\in\N\:\forall^{0}\eps$''.
\item Condition \eqref{eq:coeffmoderate} of being weakly $\rho$-moderate
represents a constrain on what coefficients $a_{n}$ we can consider
in a hyperseries. For example, if $(a_{n})_{n\in\N}$ is a sequence
of complex numbers satisfying $|a_{n}|<p(n)$, where $p\in\R[x]$
is a polynomial, then $p(n)\leq\rho_{\eps}^{-nQ}$ for all $\eps$
sufficiently small and for all $n\in\N$ if $Q\geq\max\left(1,\max\left\{ -\frac{\log n}{p(n)\log\rho_{\eps}}:n<N_{1}\right\} \right)$,
where $\frac{\log n}{p(n)}\leq1$ for all $n\geq N_{1}$ and $-\frac{1}{\log\rho_{\eps}}\leq1$.
Hence $(a_{n})_{n,\eps}\in\left(\R^{\N\times I}\right)_{\rho}$ is
weakly $\rho$-moderate. On the contrary, we cannot have $n^{n}\leq\rho_{\eps}^{-nQ-R}=\rho_{\eps}^{-R}\left(\frac{1}{\rho_{\eps}^{Q}}\right)^{n}$
for all $n\in\N$. Similarly $(n!)_{n\in\N}$ is not weakly $\rho$-moderate
and hence our theory does not apply to a ``hyperseries'' of the
form $\hypersum{}{\rho}n!z^{n}$.
\item Let $a_{n\eps}=\rho_{\eps}^{\frac{n+1}{\eps}}$, so that $[a_{n\eps}]_{{\rm c}}=0$.
The corresponding radius of convergence is $r_{\eps}=\lim_{n\rightarrow+\infty}|a_{n\eps}|^{\nicefrac{1}{n}}=\rho_{\eps}^{\nicefrac{1}{\eps}}$
which is not $\rho$-moderate. In general, if $\text{rad}(a_{n})_{{\rm c}}=[r_{\eps}]=:r\in\RCrealinfty{\rho}$,
we can have different behavior on different subpoints, e.g. $r|_{L_{1}}=+\infty$,
$r|_{L_{2}}\in\RCrealrho$, $r|_{L_{3}}$ is non $\rho$-moderate,
etc., where $L_{i}\subseteq_{0}I$. This behavior is studied in Thm.
\ref{lem:radconvergence} below.
\item Let $(a_{n})_{{\rm c}}=[a_{n\eps}]_{{\rm c}}\in\RCcomplexrho_{{\rm c}}$
and assume that for all $\eps$ there exists $r_{\eps}:=\left(\lim_{n\rightarrow+\infty}|a_{n\eps}|^{\nicefrac{1}{n}}\right)^{-1}$
such that $r:=[r_{\eps}]\in\RCrealrho$. Then, we have $\hyperlim{}{\rho}|a_{n}|^{\nicefrac{1}{n}}=\frac{1}{r}$
and $r=\text{rad}(a_{n})_{{\rm c}}\in\RCrealrho$.
\end{enumerate}
\end{rem}

The following lemma represents a useful tool to deal with radius of
convergence. It essentially states that the radius of convergence
cannot be zero (but it can be an invertible infinitesimal), it is
finite or it equals $+\infty$ at least on some subpoint, and behaves
like a supremum (see also \cite{Mukhammadiev2021}).
\begin{thm}
\label{thm:radofconvergence}Let $(a_{n})_{{\rm c}}\in\RCcomplexrho_{{\rm c}}$
and $r=[r_{\eps}]=\emph{rad}(a_{n})_{{\rm c}}\in\RCrealinfty{\rho}$,
then we have
\begin{enumerate}
\item \label{enu:radpositive}$r>0$;
\item \label{enu:radinfinity}$r<+\infty$ or $r\sbpt{=}+\infty$;
\item \label{enu:densityofR}If $q\in\RCrealrho$ and $q<r$, then $\exists s\in\RCrealrho_{>0}:q<s\le r$.
\end{enumerate}
\end{thm}

\begin{proof}
\ref{enu:radpositive}: Assume that $|a_{n\eps}|\leq\rho_{\eps}^{-nQ-R}$
for all $\eps\leq\eps_{0}$ and for all $n\in\N$. Then $\limsup_{n\rightarrow+\infty}|a_{n\eps}|^{\nicefrac{1}{n}}\leq\lim_{n\rightarrow+\infty}\rho_{\eps}^{-Q-\frac{R}{n}}=\rho_{\eps}^{-Q}$,
i.e. $r_{\eps}\geq\rho_{\eps}^{Q}$.

\noindent\ref{enu:radinfinity}: Set $L:=\left\{ \eps\in I:r_{\eps}=+\infty\right\} $.
If $L\subseteq_{0}I$, then $r=_{L}+\infty$. Otherwise $(0,\eps_{0}]\cap L=\emptyset$
for some $\eps_{0}$, i.e. $r_{\eps}<+\infty$ for all $\eps\leq\eps_{0}$.

\noindent\ref{enu:densityofR}: Assume that $r>q$. Take representative
$[r_{\eps}]=r$ and $[q_{\eps}]=q$. By \cite[Lem.~3]{NuGi24a}, we
can find $m\in\N$ such that $r_{\eps}>q_{\eps}+\rho_{\eps}^{m}$
for small $\eps$. For all $\eps$, take $s_{\eps}:=\min\left(r_{\eps},q_{\eps}+2\rho_{\eps}^{m}\right)$.
Then $s:[s_{\eps}]\in\RCrealrho_{>0}$ and $q<s\leq r$.
\end{proof}
Even if the radius of convergence of the exponential hyperseries is
$\text{rad}\left(\frac{1}{n!}\right)_{\mathrm{c}}=+\infty$, we have
that $e^{z}=\hypersum{}{\rho}\frac{z^{n}}{n!}\in\RCcomplexrho$ implies
$|z|\leq\log\left(\diff\rho^{-R}\right)$ for some $R\in\N$: in other
words, the constraint to get a $\rho$-moderate number implies that
even if $\hypersum{}{\rho}\frac{z^{n}}{n!}$ converges at $z$, the
exponential HPS does \textit{not} converge in the entire $B\left(z\right)_{r}=\RCcomplexrho$,
where $r:=\text{rad}\left(\frac{1}{n!}\right)_{\mathrm{c}}=+\infty$.

In several steps below, we will also see that it is also very useful
to have that the generalized number $\left[\sum_{n=N_{\eps}}^{M_{\eps}}a_{n\eps}(z_{\eps}-c_{\eps})^{n}\right]\in\RCcomplexrho$,
where $N_{\eps}\in\N$ and $M_{\eps}\in\N\cup\{+\infty\}$ are $\rho$-moderate
nets, is well-defined as a function of $z=[z_{\eps}]$.

Finally, in all our examples, if the HPS $\hypersum{}{\rho}a_{n}(z-c)^{n}\in\RCcomplexrho$
converges, then it converges exactly to $\left[\sum_{n=0}^{+\infty}a_{n\eps}(z_{\eps}-c_{\eps})^{n}\right]\in\RCcomplexrho$.
These observations motivate the following definition of \textit{set
of convergence}:
\begin{defn}
\label{def:setofconvergence} Let $(a_{n})_{{\rm c}}\in\RCcomplexrho_{{\rm c}}$
and $c\in\RCcomplexrho$. The set of convergence
\[
\setconv{\rho}{}\left((a_{n})_{{\rm c}},c\right)
\]
is the set of all $z\in\RCcomplexrho$ satisfying
\begin{enumerate}
\item \label{enu:radofconvergence}$|z-c|<\text{rad}(a_{n})_{{\rm c}}$,
\end{enumerate}
and such that there exist representatives $[z_{\eps}]=z$, $[a_{n\eps}]_{{\rm c}}=(a_{n})_{{\rm c}}$,
and $[c_{\eps}]=c$ satisfying the following conditions:
\begin{enumerate}[resume]
\item \label{enu:formalHPS}$\left[a_{n\eps}(z_{\eps}-c_{\eps})^{n}\right]_{{\rm s}}\in\RCcomplexud{\rho}{}\left\llbracket z-c\right\rrbracket $,
i.e. we have a formal HPS;
\item \label{enu:epsHPS}$\hypersum{}{\rho}a_{n}(z-c)^{n}=\left[\sum_{n=0}^{+\infty}a_{n\eps}(z_{\eps}-c_{\eps})^{n}\right]\in\RCcomplexrho$;
\item \label{enu:independence}For all representatives $[\hat{z}_{\eps}]=z$,
the derivative net is $\rho$-moderate:
\begin{equation}
\left(\frac{\text{d}}{\text{d}z}\left(\sum_{n=0}^{+\infty}a_{n\eps}(z-c_{\eps})^{n}\right)_{z=\hat{z}_{\eps}}\right)\in\CC_{\rho}.\label{eq:independentrepresentatives}
\end{equation}
\end{enumerate}
\end{defn}

Note that condition \ref{enu:formalHPS} is necessary because in \ref{enu:epsHPS}
we use a HPS; on the other hand, condition \ref{enu:epsHPS} and \ref{enu:independence}
state that we have a GHF defined by the net of holomorphic functions
$\left(\sum_{n=0}^{+\infty}a_{n\eps}(z-c_{\eps})^{n}\right)$. As
for GHF, see \cite[Thm. 26]{NuGi24a}, condition \ref{enu:independence}
only on first derivative suffices to prove that we actually have a
GHF in $z$.

We immediately note that $z\in\setconv{\rho}{}\left((a_{n})_{{\rm c}},c\right)$
if and only if $z-c\in\setconv{\rho}{}\left((a_{n})_{{\rm c}},0\right)$,
and because of this property without loss of generality we will frequently
assume $c=0$. We also note that condition \ref{enu:epsHPS} states
that the HPS $\hypersum{}{\rho}a_{n}z^{n}$ converges, and it does
exactly to the generalized number generated by $\sum_{n=0}^{+\infty}a_{n\eps}(z_{\eps}-c_{\eps})^{n}$.

As expected, for HPS the set of convergence is never trivial:
\begin{thm}
\label{thm:nonemptysetofconvergence}Let $(a_{n})_{c}\in\RCcomplexrho_{{\rm c}}$
and $c\in\RCcomplexrho$. Then 
\[
\exists q\in\N:B_{\diff\rho^{q}}(c)\subseteq\setconv{\rho}{}\left((a_{n})_{{\rm c}},c\right).
\]
\end{thm}

\begin{proof}
From Thm. \ref{thm:radofconvergence}\ref{enu:radpositive}, we have
$r:=\text{rad}(a_{n})_{{\rm c}}\geq\diff\rho^{q_{1}}$ for some $q_{1}\in\N$.
We also have $|a_{n\eps}|\leq\rho_{\eps}^{-nQ-R}$ from \ref{eq:coeffmoderate}.
We have to find $q\in\N_{\geq q_{1}}$ so that $B_{\diff\rho^{q}}(c)\subseteq\setconv{\rho}{}\left((a_{n})_{{\rm c}},c\right)$.
To prove property Def. \ref{def:setofconvergence}\ref{enu:formalHPS},
for $N_{\eps}$, $M_{\eps}\in\N$ and for $\eps$ small, we estimate
\[
\left|\sum_{n=N_{\eps}}^{M_{\eps}}a_{n\eps}(z_{\eps}-c_{\eps})^{n}\right|\leq\sum_{n=N_{\eps}}^{M_{\eps}}\rho_{\eps}^{-nQ-R}\rho_{\eps}^{nq}=\rho_{\eps}^{-R}\sum_{n=N_{\eps}}^{M_{\eps}}\rho_{\eps}^{-nQ+nq}.
\]
Therefore, taking $q:=\max\left(1+Q,q_{1}\right)$, we get
\[
\left|\sum_{n=N_{\eps}}^{M_{\eps}}a_{n\eps}(z_{\eps}-c_{\eps})^{n}\right|\leq\rho_{\eps}^{-R}\sum_{n=N_{\eps}}^{M_{\eps}}\rho_{\eps}^{n}\leq\frac{\rho_{\eps}^{-R}}{1-\rho_{\eps}},
\]
and this proves Def.~\ref{def:setofconvergence}\ref{enu:formalHPS}.
Similarly, we have 
\begin{align*}
\left|\sum_{n=0}^{M_{\eps}}a_{n\eps}(z_{\eps}-c_{\eps})^{n}-\sum_{n=0}^{+\infty}a_{n\eps}(z_{\eps}-c_{\eps})^{n}\right| & \leq\sum_{n=M_{\eps}+1}^{+\infty}\rho_{\eps}^{n}\leq\frac{\rho_{\eps}^{M_{\eps}+1}}{1-\rho_{\eps}}.
\end{align*}
Since $\hyperlim{\rho}{}\diff\rho^{M+1}=0$, this proves Def. \ref{def:setofconvergence}\ref{enu:epsHPS}.
Finally, for all representatives $[\hat{z}_{\eps}]=z$, we have 
\begin{align*}
\frac{\text{d}}{\text{d}z}\left(\sum_{n=0}^{+\infty}a_{n\eps}(z-c_{\eps})^{n}\right)_{z=\hat{z}_{\eps}} & =\sum_{n=0}^{+\infty}na_{n\eps}(\hat{z}_{\eps}-c_{\eps})^{n-1}
\end{align*}
and hence 
\begin{align*}
\left|\frac{\text{d}}{\text{d}z}\left(\sum_{n=0}^{+\infty}a_{n\eps}(z-c_{\eps})^{n}\right)_{z=\hat{z}_{\eps}}\right| & \leq\sum_{n=0}^{+\infty}n\rho_{\eps}^{-nQ-R}\rho_{\eps}^{\left(n-1\right)q}\\
 & =\rho_{\eps}^{-R-q}\sum_{n=0}^{+\infty}n\rho_{\eps}^{(q-Q)n}\\
 & =\rho_{\eps}^{-R-q}\frac{\rho_{\eps}^{q-Q}}{(1-\rho_{\eps}^{q-Q})^{2}}.
\end{align*}
In the last step, we use $q\geq Q+1$ and the series $\sum_{n=0}^{+\infty}ny^{n}=\frac{y}{(1-y)^{2}}$
for $|y|<1$. This proves property \ref{enu:independence}.
\end{proof}
We can now prove independence from representatives both in Def. \ref{def:setofconvergence}
and in Def. \ref{def:formalHPS}:
\begin{thm}
\label{thm:independencerepresentatives}Let $(a_{n})_{{\rm c}}=[a_{n\eps}]_{{\rm c}}=[\hat{a}_{n\eps}]_{{\rm c}}\in\RCcomplexrho_{{\rm c}}$,
$z=[z_{\eps}]=[\hat{z}_{\eps}]$, $c=[c_{\eps}]=[\hat{c}_{\eps}]\in\RCcomplexrho$.
Assume that $z\in\setconv{\rho}{}\left((a_{n})_{{\rm c}},c\right)$.
Then,
\begin{enumerate}
\item \label{enu:independence1}The nets $(a_{n\eps})_{n,\eps}$, $(z_{\eps})$,
and $(c_{\eps})$ also satisfy all the conditions of Def.~\ref{def:setofconvergence}
of set of convergence.
\item \label{enu:independenceequal}$\left[a_{n\eps}(z_{\eps}-c_{\eps})^{n}\right]_{{\rm s}}=\left[\hat{a}_{n\eps}(\hat{z}_{\eps}-\hat{c}_{\eps})^{n}\right]_{{\rm s}}$,
where the equality is in $\RCcomplexrho_{{\rm s}}$ (see Def.~\ref{def:hyperseries}).
\end{enumerate}
\end{thm}

\begin{proof}
\ref{enu:independence1}: Since we have similar steps for several
claims, let $N_{\eps}\in\N$ and $M_{\eps}\in\N\cup\{+\infty\}$,
so that a term of the form $\sum_{n=N_{\eps}}^{M_{\eps}}b_{n\eps}$
represents both the ordinary series $\sum_{n=0}^{+\infty}b_{n\eps}$
or the finite sum $\sum_{n=N_{\eps}}^{M_{\eps}}b_{n\eps}$. From Def.~\ref{def:setofconvergence}
of set of convergence, we get the existence of representatives $[\hat{z}_{\eps}]=z\in\RCcomplexrho$,
$[\hat{a}_{n\eps}]_{{\rm c}}=(a_{n})_{{\rm c}}$ and $[\hat{c}_{\eps}]=c$
satisfying Def. \ref{def:setofconvergence}. Set $y_{\eps}:=z_{\eps}-c_{\eps}$,
$y:=[y_{\eps}]$, $\hat{y}_{\eps}:=\hat{z}_{\eps}-\hat{c}_{\eps}$
and $\hat{y}:=[\hat{y}_{\eps}]$. Let $r:=[r_{\eps}]=\text{rad}(a_{n})_{{\rm c}}$
be the radius of convergence. Take $s\in\RCrealrho$ satisfying $|\hat{y}|<s\leq r$
and a representative $[s_{\eps}]=s$ such that $|\hat{y}_{\eps}|<s_{\eps}\leq r_{\eps}$
for all $\eps$ small. Set $b_{n\eps}:=a_{n\eps}-\hat{a}_{n\eps}$
and $h_{\eps}:=y_{\eps}-\hat{y}_{\eps}$. We have
\begin{align}
\sum_{n=N_{\eps}}^{M_{\eps}}a_{n\eps}y_{\eps}^{n} & =\sum_{n=N_{\eps}}^{M_{\eps}}\left(\hat{a}_{n\eps}+b_{n\eps}\right)\left(\hat{y}_{\eps}+h_{\eps}\right)^{n}\nonumber \\
 & =\sum_{n=N_{\eps}}^{M_{\eps}}\hat{a}_{n\eps}\left(\hat{y}_{\eps}+h_{\eps}\right)^{n}+\sum_{n=N_{\eps}}^{M_{\eps}}b_{n\eps}\left(\hat{y}_{\eps}+h_{\eps}\right)^{n}\label{eq:negligibledifference-1}
\end{align}
and 
\begin{align}
\sum_{n=N_{\eps}}^{M_{\eps}}na_{n\eps}y_{\eps}^{n-1} & =\sum_{n=N_{\eps}}^{M_{\eps}}n\left(\hat{a}_{n\eps}+b_{n\eps}\right)\left(\hat{y}_{\eps}+h_{\eps}\right)^{n-1}\nonumber \\
 & =\sum_{n=N_{\eps}}^{M_{\eps}}n\hat{a}_{n\eps}\left(\hat{y}_{\eps}+h_{\eps}\right)^{n-1}+\sum_{n=N_{\eps}}^{M_{\eps}}nb_{n\eps}\left(\hat{y}_{\eps}+h_{\eps}\right)^{n-1}.\label{eq:negligibledifference}
\end{align}
Since $h:=[h_{\eps}]=0$, we also have $|\hat{y}_{\eps}|+|h_{\eps}|<s_{\eps}\leq r_{\eps}$
for all $\eps$ small. For the same $\eps$, assume that $|b_{n\eps}|\leq\rho_{\eps}^{np+q}$
for fixed arbitrary $p$, $q\in\N$. We first consider the second
summand of \eqref{eq:negligibledifference-1} and \eqref{eq:negligibledifference}:
\begin{align*}
\left|\sum_{n=N_{\eps}}^{M_{\eps}}b_{n\eps}\left(\hat{y}_{\eps}+h_{\eps}\right)^{n}\right| & \leq\sum_{n=N_{\eps}}^{M_{\eps}}\rho_{\eps}^{np+q}s_{\eps}^{n}\leq\rho_{\eps}^{q}\sum_{n=N_{\eps}}^{M_{\eps}}\left(\rho_{\eps}^{p}s_{\eps}\right)^{n}\text{ and}\\
\left|\sum_{n=N_{\eps}}^{M_{\eps}}nb_{n\eps}\left(\hat{y}_{\eps}+h_{\eps}\right)^{n-1}\right| & \leq\sum_{n=N_{\eps}}^{M_{\eps}}n\rho_{\eps}^{np+q}s_{\eps}^{n-1}\leq\rho_{\eps}^{q}\sum_{n=N_{\eps}}^{M_{\eps}}n\left(\rho_{\eps}^{p}s_{\eps}\right)^{n-1}.
\end{align*}
Since $s\in\RCrealrho$, we can take $p\in\N$ sufficiently large
so that $\rho_{\eps}^{p}s_{\eps}<1$. This implies 
\begin{align*}
\left|\sum_{n=N_{\eps}}^{M_{\eps}}b_{n\eps}\left(\hat{y}_{\eps}+h_{\eps}\right)^{n}\right| & \leq\frac{\rho_{\eps}^{q}}{\left(1-\rho_{\eps}^{p}s_{\eps}\right)}\\
\text{and }\left|\sum_{n=N_{\eps}}^{M_{\eps}}nb_{n\eps}\left(\hat{y}_{\eps}+h_{\eps}\right)^{n-1}\right| & \leq\frac{\rho_{\eps}^{q}}{\left(1-\rho_{\eps}^{p}s_{\eps}\right)^{2}}.
\end{align*}
Thereby, for $q\rightarrow+\infty$, this summand defines a negligible
net. For the first summand of \eqref{eq:negligibledifference}, we
have
\begin{align*}
\left|\sum_{n=N_{\eps}}^{M_{\eps}}\hat{a}_{n\eps}\left(\hat{y}_{\eps}+h_{\eps}\right)^{n}-\sum_{n=N_{\eps}}^{M_{\eps}}\hat{a}_{n\eps}\hat{y}_{\eps}^{n}\right| & \qquad\qquad\qquad\qquad\qquad
\end{align*}
\begin{align}
\qquad\qquad & =\left|\sum_{n=N_{\eps}}^{M_{\eps}}\hat{a}_{n\eps}\left(\sum_{k=0}^{n}\binom{n-1}{k}\hat{y}_{\eps}^{n-k}h_{\eps}^{k}\right)-\sum_{n=N_{\eps}}^{M_{\eps}}\hat{a}_{n\eps}\hat{y}_{\eps}^{n}\right|\nonumber \\
 & =\left|\sum_{n=N_{\eps}}^{M_{\eps}}\hat{a}_{n\eps}\left(\sum_{k=1}^{n-1}\binom{n-1}{k}\hat{y}_{\eps}^{n-k}h_{\eps}^{k}\right)\right|\nonumber \\
 & =\left|h_{\eps}\right|\left|\sum_{n=N_{\eps}}^{M_{\eps}}\hat{a}_{n\eps}\left(\sum_{k=1}^{n-1}\binom{n-1}{k}\hat{y}_{\eps}^{n-k}h_{\eps}^{k-1}\right)\right|\label{eq:negligibledifference2-1}
\end{align}
and

\begin{align*}
\left|\sum_{n=N_{\eps}}^{M_{\eps}}n\hat{a}_{n\eps}\left(\hat{y}_{\eps}+h_{\eps}\right)^{n-1}-\sum_{n=N_{\eps}}^{M_{\eps}}n\hat{a}_{n\eps}\hat{y}_{\eps}^{n-1}\right| & \qquad\qquad\qquad\qquad\qquad
\end{align*}
\begin{align}
\qquad\qquad & =\left|\sum_{n=N_{\eps}}^{M_{\eps}}n\hat{a}_{n\eps}\left(\sum_{k=0}^{n-1}\binom{n-1}{k}\hat{y}_{\eps}^{(n-1)-k}h_{\eps}^{k}\right)-\sum_{n=N_{\eps}}^{M_{\eps}}n\hat{a}_{n\eps}\hat{y}_{\eps}^{n-1}\right|\nonumber \\
 & =\left|\sum_{n=N_{\eps}}^{M_{\eps}}n\hat{a}_{n\eps}\left(\sum_{k=1}^{n-1}\binom{n-1}{k}\hat{y}_{\eps}^{(n-1)-k}h_{\eps}^{k}\right)\right|\nonumber \\
 & =\left|h_{\eps}\right|\left|\sum_{n=N_{\eps}}^{M_{\eps}}n\hat{a}_{n\eps}\left(\sum_{k=1}^{n-1}\binom{n-1}{k}\hat{y}_{\eps}^{(n-1)-k}h_{\eps}^{k-1}\right)\right|.\label{eq:negligibledifference2}
\end{align}
Thereby, the right hand side of \eqref{eq:negligibledifference2-1}
and \eqref{eq:negligibledifference2} is negligible because $h:=[h_{\eps}]=0$.
We can hence state that 
\begin{equation}
\left(\sum_{n=N_{\eps}}^{M_{\eps}}a_{n\eps}y_{\eps}^{n}\right)\sim_{\rho}\left(\sum_{n=N_{\eps}}^{M_{\eps}}\hat{a}_{n\eps}\left(\hat{y}_{\eps}+h_{\eps}\right)^{n}\right)\text{ and }\label{eq:moderateness}
\end{equation}
\[
\left(\sum_{n=N_{\eps}}^{M_{\eps}}na_{n\eps}y_{\eps}^{n-1}\right)\sim_{\rho}\left(\sum_{n=N_{\eps}}^{M_{\eps}}n\hat{a}_{n\eps}\left(\hat{y}_{\eps}+h_{\eps}\right)^{n-1}\right).
\]
These prove that $[a_{n\eps}\cdot y_{\eps}^{n}]_{s}\in\RCcomplexrho_{{\rm s}}$,
the moderateness of $\left(\sum_{n=0}^{+\infty}a_{n\eps}y_{\eps}^{n}\right)$,
and Def. \ref{def:setofconvergence}\ref{enu:independence}. Hence,
we also proved claim \ref{enu:independenceequal}. Moreover, from
\eqref{eq:moderateness}, we have $\hypersum{}{\rho}[\hat{a}_{n\eps}]\cdot[\hat{y}_{\eps}]^{n}$
converges to $\left[\sum_{n=0}^{+\infty}\hat{a}_{n\eps}\hat{y}_{\eps}^{n}\right]=\left[\sum_{n=0}^{+\infty}a_{n\eps}y_{\eps}^{n}\right]\in\RCcomplexrho$.
\end{proof}
In the following examples, we start studying geometric hyperseries,
which are convergent HPS:
\begin{example}[Geometric HPS]
\label{exa:geometricHPS}Assume that $z\in\RCcomplexrho$ and $|z|<1$.
We have
\[
\left[\left|\sum_{n=0}^{N_{\eps}}z_{\eps}^{n}\right|\right]\leq\left[\left|\frac{1-\left|z_{\eps}^{N_{\eps}+1}\right|}{1-z_{\eps}}\right|\right]\leq\frac{2}{1-z}\in\RCcomplexrho.
\]
This shows that $(z^{n})_{n}=[z_{\eps}^{n}]\in\RCcomplexrho_{{\rm s}}$.
Moreover, from Ex. \ref{exa:hypseries}\ref{enu:geometryconverge},
we have 
\[
\hypersum{}{\rho}z^{n}=\frac{1}{1-z}.
\]
Since coefficients $a_{n\eps}=1$, we have $[a_{n\eps}]_{{\rm c}}\in\RCcomplexrho_{{\rm c}}$,
$\text{rad}(a_{n})_{{\rm c}}=1$, and $\setconv{\rho}{}\left((1)_{{\rm c}},0\right)\subseteq B_{1}(0)$.
Now, take $z=[z_{\eps}]\in B_{1}(0)$, with $|z_{\eps}|<1$ for all
$\eps$, we have 
\[
\hypersum{}{\rho}z^{n}=\left[\sum_{n=0}^{+\infty}z_{\eps}^{n}\right]\in\RCcomplexrho
\]
and if $[\hat{z}_{\eps}]=z$ is another representative and $k\in\N_{>0}$,
then $\hat{z}_{\eps}\in\Eball_{1}(0)$ and we get $\left(\sum_{n=k}^{+\infty}n\hat{z}_{\eps}^{n-1}\right)=\left(\frac{1}{\left(1-\hat{z}_{\eps}\right)^{2}}\right)\in\CC_{\rho}$
because $|1-z|>0$ is invertible. In particular, the convergence of
the geometric hyperseries $\hypersum{}{\rho}\diff\rho^{n}=\frac{1}{1-\diff\rho}$.
\end{example}

\begin{example}[Exponential]
\label{exa:exponential}We clearly have $\left(\frac{1}{n!}\right)_{{\rm c}}\in\RCcomplexrho_{{\rm c}}$
and $\text{rad}\left(\frac{1}{n!}\right)_{{\rm c}}=+\infty$, i.e.
we have coefficients for an HPS with infinite radius of convergence.
Set 
\[
C:=\left\{ z\in\RCcomplexrho\mid\exists K\in\N:|z|<-K\log\diff\rho\right\} .
\]
For all $z=[z_{\eps}]\in C$ and all $N_{\eps}$, $M_{\eps}\in\N$,
we have $\left|\sum_{n=N_{\eps}}^{M_{\eps}}\frac{z_{\eps}^{n}}{n!}\right|\leq e^{|z_{\eps}|}\leq\rho_{\eps}^{-K}$for
$\eps$ small, and this shows that $\left(\frac{z^{n}}{n!}\right)_{n}\in\RCcomplexud{\rho}{}\left\llbracket z\right\rrbracket $,
i.e. for all $z\in C$, we have a formal HPS. We finally want to prove
that $C=\setconv{\rho}{}\left(\left(\frac{1}{n!}\right)_{{\rm c}},0\right)$.
The inclusion $\supseteq$ follows directly from Def.~\ref{def:setofconvergence}.
If $z=[z_{\eps}]\in C$, then $\left(\sum_{n=0}^{+\infty}\frac{z_{\eps}^{n}}{n!}\right)$
and $\left(\sum_{n=0}^{+\infty}n\frac{z_{\eps}^{n-1}}{n!}\right)\in\RCcomplexrho$.
To prove Def.~\ref{def:setofconvergence}\ref{enu:epsHPS}, assume
that $|z_{\eps}|<-K\log\rho_{\eps}=:M_{\eps}$ for all $\eps$ and
set $M:=[M_{\eps}]\in\RCrealrho$. Take $N=\left[N_{\eps}\right]\in\hypNr$
such that $\frac{M}{N+1}<\frac{1}{2}$, so that $\frac{M^{n+1}}{(n+1)!}<\frac{1}{2^{n+1}}$
and hence $\left|\sum_{n=N_{\eps}+1}^{+\infty}\frac{z_{\eps}^{n}}{n!}\right|\leq\sum_{n\geq N_{\eps}}\frac{1}{2^{n}}\rightarrow0$
as $N\rightarrow+\infty$, $N\in\hypNr$. Similarly, we can consider
trigonometric functions whose set of convergence is the whole of $\RCcomplexrho$.
\end{example}

\begin{example}[Dirac delta]
\label{exa:diracdelta}Recall from \cite[Sec. 3]{NuGi24a} that we
can define a Dirac delta embedded as a GHF, by $\delta_{\eps}(z):=\left(\delta\ast\mu_{\eps}\right)(z)=\mu_{\eps}(z)=\rho_{\eps}^{-2}\mu\left(\rho_{\eps}^{-1}z\right)$
where $\mu\in\mathcal{C}^{\infty}(\R^{2})$ is an entire function
such that $\int_{\R^{2}}\mu(x)\diff x=1$ ($\mu$ is called entire
mollifier in \cite{NuGi24a}). We have that $\left(\frac{\delta^{(n)}(0)}{n!}\right)_{{\rm c}}\in\RCcomplexrho_{{\rm c}}$
and $\text{rad}\left(\frac{\delta^{(n)}(0)}{n!}\right)_{{\rm c}}=+\infty$.
In fact, for $n\in\N$, we have $\left|\frac{\delta_{\eps}^{(n)}(0)}{n!}\right|=\left|\frac{\mu^{(n)}(0)}{n!}\rho_{\eps}^{-2n-2}\right|\leq\frac{1}{n!}\rho_{\eps}^{-2n-2}\leq\rho_{\eps}^{-2n-2}$
and 
\[
r_{\eps}^{-1}=\limsup_{n\rightarrow+\infty}\left|\frac{\delta_{\eps}^{(n)}(0)}{n!}\right|^{\nicefrac{1}{n}}=\limsup_{n\rightarrow+\infty}\rho_{\eps}^{-2(1+\nicefrac{1}{n})}\left|\frac{\mu^{(n)}(0)}{n!}\right|^{\nicefrac{1}{n}}=\rho_{\eps}^{-2(1+\nicefrac{1}{n})}\cdot0=0.
\]
For all $z=[z_{\eps}]\in\RCcomplexrho$ and all $N_{\eps}$, $M_{\eps}\in\N$,
we have
\[
\left|\sum_{n=N_{\eps}}^{M_{\eps}}\frac{\delta_{\eps}^{(n)}(0)}{n!}z_{\eps}^{n}\right|\leq\rho_{\eps}^{-2}\cdot\sum_{n=0}^{+\infty}\frac{|\mu^{(n)}(0)|}{n!}\left|\rho_{\eps}^{-1}z_{\eps}\right|^{n}=\rho_{\eps}^{-2}\mu\left(|\rho_{\eps}^{-1}z_{\eps}|\right)\in\CC_{\rho}.
\]
This proves that $\left(\frac{\delta^{(n)}(0)}{n!}\right)_{n}\in\RCcomplexud{\rho}{}\left\llbracket z\right\rrbracket $,
i.e.~we always have a formal HPS. It remains to prove that $\setconv{\rho}{}\left(\left(\frac{\delta^{(n)}(0)}{n!}\right)_{{\rm c}},0\right)=\RCcomplexrho$:
for all $z=[z_{\eps}]\in\RCcomplexrho$
\[
\sum_{n=0}^{N}\frac{\delta^{(n)}(0)}{n!}z^{n}=\left[\rho_{\eps}^{-2}\sum_{n=0}^{N_{\eps}}\frac{\mu^{(n)}(0)}{n!}\left(\rho_{\eps}^{-1}z_{\eps}\right)^{n}\right]=\delta(z)-\diff\rho^{-2}\left[\mu^{(N_{\eps}+1)}(\hat{z}_{\eps})\frac{z_{\eps}^{N_{\eps}+1}}{(N_{\eps}+1)!}\right]
\]
where the existence of $\hat{z}_{\eps}\in B_{\left|z_{\eps}\right|}(0)$
is derived from Taylor's formula of $\mu$ and $\mu^{(N_{\eps}+1)}(\hat{z}_{\eps})\frac{z_{\eps}^{N_{\eps}+1}}{(N_{\eps}+1)!}$
is the Lagrange remainder. We obtain 
\[
\left|\sum_{n=0}^{N}\frac{\delta^{(n)}(0)}{n!}z^{n}-\delta(z)\right|\leq\diff\rho^{-2}\left|\mu^{(N+1)}(\hat{z})\right|\frac{|z|^{N+1}}{\left(N+1\right)!}.
\]
Using Stirling's approximation, we have $\frac{|z|^{N+1}}{\left(N+1\right)!}\leq2\left(\frac{|z|e}{N}\right)^{N}\leq\diff\rho^{N}$,
for all $N\in\hypNr$ such that $N>|z|e\diff\rho^{-1}$. Since $\lim_{N\in\hypNr}\diff\rho^{N}=0$,
this proves the claim.
\end{example}

When we say that a HPS $\hypersum{}{\rho}a_{n}(z-c)^{n}$ is \textit{convergent},
we already assume that its coefficients are correctly chosen and that
the point $z$ is in the set of convergence, as stated in the following:
\begin{defn}
\label{def:convergentHPS}We say that $\hypersum{}{\rho}a_{n}(z-c)^{n}$
is a \textit{convergent} HPS, if
\begin{enumerate}
\item $(a_{n})_{{\rm c}}\in\RCcomplexrho_{{\rm c}}$ are coefficients for
HPS;
\item $z\in\setconv{\rho}{}\left((a_{n})_{{\rm c}},c\right)$.
\end{enumerate}
\end{defn}

We now study absolute convergence of HPS, and sharply boundedness
of the summands of a HPS. We first show that the hypersequence $\left(a_{n}\left(z-c\right)^{n}\right)_{n\in\hypNr}$
of the terms of a HPS is sharply bounded:
\begin{lem}
\label{lem:boundedsummands}Let $(a_{n})_{{\rm c}}\in\RCcomplexrho_{{\rm c}}$
and $c\in\RCcomplexrho$.
\begin{enumerate}
\item \label{enu:boundedsum}If $\hypersum{}{\rho}a_{n}(z-c)^{n}$ is a
convergent HPS, then
\[
\exists K\in\hypNr\;\forall n\in\hypNr:\ |a_{n}(z-c)^{n}|<K.
\]
\item \label{enu:eventuallyBounded}For all representatives $(a_{n})_{{\rm c}}=\left[a_{n\eps}\right]_{{\rm c}}$and
$[c_{\eps}]=c$, there exists $\delta\in\RCrealrho_{>0}$ such that
for all $z=[z_{\eps}]\in B_{\delta}(c)$, we have
\begin{equation}
\exists K=[K_{\eps}]\in\RCrealrho_{>0}\,\forall^{0}\eps\,\forall n\in\N:\ \left|a_{n\eps}(z_{\eps}-c_{\eps})^{n}\right|<K_{\eps}.\label{eq:eventuallybounded}
\end{equation}
We say that $(a_{n}(z-c)^{n})_{n\in\hypNr}$ is \textup{eventually
bounded} (in $\RCcomplexrho{}_{{\rm c}}$) if there exist representatives
$(a_{n})_{{\rm c}}=\left[a_{n\eps}\right]_{{\rm c}}$, $[z_{\eps}]=z$,
and $[c_{\eps}]=c$ such that \eqref{eq:eventuallybounded} holds.
\end{enumerate}
\end{lem}

\begin{proof}
\ref{enu:boundedsum}: We recall that because of the definition of
formal HPS Def.~\ref{def:formalHPS}, the term $a_{n}(z-c)^{n}\in\RCcomplexrho$
is well-defined for all $n\in\hypNr$. Set $\hat{z}:=z-c$, i.e.~without
loss of generality we can assume $c=0$. Since $\hypersum{}{\rho}a_{n}\hat{z}^{n}$
converges, from Lem. \ref{lem:limitseq}, we have
\begin{equation}
\exists N\in\hypNr\,\forall n\in\hypNr:|a_{n}\hat{z}^{n}|<1.\label{eq:boundedHPS}
\end{equation}
Let us consider an arbitrary $n\in\hypNr$. We have either $n\geq N$
or $n<_{L}N$ for some $L\subseteq_{0}I$. In the latter case, $|a_{n}\hat{z}^{n}|\leq_{L}s:=\sum_{n=0}^{N-1}|a_{n}\hat{z}^{n}|<\max\left(s+1,1\right)=:K$.
From this and \eqref{lem:boundedsummands}, the claim follows.

\noindent\ref{enu:eventuallyBounded}: Since $(a_{n})_{c}\in\RCcomplexud{\rho}{}_{c}$,
we have $\forall^{0}\eps\forall n\in\N:|a_{n\eps}|\leq\rho_{\eps}^{-nQ-R}$.
Therefore, for $\delta:=\diff\rho^{Q}$, we have $\left|a_{n\eps}(z_{\eps}-c_{\eps})^{n}\right|\leq\rho_{\eps}^{-nQ-R}\rho_{\eps}^{nQ}=\rho_{\eps}^{-R}=:K$.
\end{proof}
Note that the adjective \emph{eventually} in \ref{enu:eventuallyBounded}
above means that only for $\eps\le\eps_{0}$ ($\eps_{0}$ depending
on $(a_{n})_{{\rm c}}=\left[a_{n\eps}\right]_{{\rm c}}$, $[z_{\eps}]=z$,
and $[c_{\eps}]=c$) the sequence $n\in\N\mapsto a_{n\eps}(z_{\eps}-c_{\eps})^{n}\in\CC$
is bounded.

Unfortunately, property \eqref{eq:eventuallybounded} does not hold
for all points $z\in\setconv{\rho}{}\left(\left(a_{n}\right)_{{\rm c}},c\right)$.
For example, consider in Example \ref{exa:diracdelta} the Dirac delta
at $c=0$ and $|z|\leq\diff\rho^{2}$: we have $\left|\frac{\delta_{\eps}^{(n)}(0)}{n!}z_{\eps}^{n}\right|=\left|\frac{\mu^{(n)}(0)}{n!}\rho_{\eps}^{-2n-2}z_{\eps}^{n}\right|\leq\rho_{\eps}^{-2}$
for all $n\in\N$ such that $\left|\frac{\mu^{(n)}(0)}{n!}\right|\leq1$.
Therefore for this $z$, $\left(\frac{\delta^{(n)}(0)}{n!}z^{n}\right)_{n\in\hypNr}$
is eventually bounded in $\RCcomplexrho_{{\rm c}}$. However, if $|z|\gg0$,
i.e.~$|z|\geq s\in\R_{>0}$, then $\left|\frac{\delta_{\eps}^{(n)}(0)}{n!}z_{\eps}^{n}\right|=\left|\frac{\mu^{(n)}(0)}{n!}\rho_{\eps}^{-2n-2}z_{\eps}^{n}\right|\geq\left|\frac{\mu^{(n)}(0)}{n!}\right|\rho_{\eps}^{-2n-2}s^{n}$
and hence condition \eqref{eq:eventuallybounded} does not hold for
any $[K_{\eps}]\in\RCrealrho$.

On the other hand, Example \ref{exa:geometricHPS} of geometric hyperseries
and Example \ref{exa:exponential} of exponential HPS are all eventually
bounded at $c=0$ and for all finite number $z$. This follows from
the property that, for $\eps$ small, the series $\sum_{n=0}^{+\infty}\left|a_{n\eps}\left(z_{\eps}-c_{\eps}\right)^{n}\right|=:K_{\eps}$
of absolute values terms converges to a $\rho$-moderate net.

The question of whether the convergence set is a sharply open set
remains an unsolved problem. Notably, the case of exponential HPS
(see Example \ref{exa:exponential}) demonstrates that the convergence
set is not, in general, a disk. A partial solution is to consider
the interior of $\setconv{\rho}{}\left(\left(a_{n}\right)_{{\rm c}},c\right)$,
which is always not empty because of Lem.~\ref{lem:boundedsummands}
and of Thm.~\ref{thm:uniformlyconvergent} below. Another result
going in this direction is the following theorem, where interior points
of the set of convergence are related to points where we have eventual
boundedness.
\begin{thm}
\label{thm:uniformlyconvergent}Let $(a_{n})_{{\rm c}}=[a_{n\eps}]_{{\rm c}}\in\RCcomplexrho_{{\rm c}}$,
$c=[c_{\eps}]\in\RCcomplexrho$, and $\hat{z}\in\setconv{\rho}{}\left(\left(a_{n}\right)_{{\rm c}},c\right)$.
If $\left(a_{n}(\hat{z}-c)^{n}\right)_{n\in\hypNr}$ is eventually
bounded, then for \textup{all $z=[z_{\eps}]\in B_{\left|\hat{z}-c\right|}(c)$
we have}
\begin{enumerate}
\item \textup{\label{enu:absolute}The HPS $\hypersum{}{\rho}a_{n}(z-c)^{n}$
converges absolutely and }$\left(a_{n}(z-c)^{n}\right)_{n\in\hypNr}$
is eventually bounded. Moreover,
\[
\hypersum{}{\rho}a_{n}(z-c)^{n}=\left[\sum_{n=0}^{+\infty}a_{n\eps}(z_{\eps}-c_{\eps})^{n}\right]\in\RCcomplexrho;
\]
\item \label{enu:uniform}For all functionally compact set $K\Subset_{\text{\emph{f}}}\overline{B_{|\hat{z}-c|}(c)}$,
\textup{the HPS $\hypersum{}{\rho}a_{n}(z-c)^{n}$ converges uniformly
on $K$.}
\item \label{enu:interiorofsetofconv}The point $z$ is a sharply interior
point of $\setconv{\rho}{}\left(\left(a_{n}\right)_{{\rm c}},c\right)$.
\end{enumerate}
\end{thm}

\begin{proof}
Without loss of generality we can assume $c=0$.

\noindent\ref{enu:absolute}: We have either $\hat{z}=_{L}0$ or
$\left|\hat{z}\right|>0$ for some $L\subseteq_{0}I$. The first case
is actually impossible because $0\leq|z|<|\hat{z}|=_{L}0$. We can
hence work only in the latter case $|\hat{z}|>0$. From Lemma \ref{lem:boundedsummands}\ref{enu:eventuallyBounded},
we have $\forall^{0}\eps\,\forall n\in\N:|a_{n\eps}\hat{z}_{\eps}^{n}|\leq K_{\eps}$.
Setting $h:=\left|\frac{z}{\hat{z}}\right|$, we have $h<1$, because
$z\in B_{|\hat{z}|}(0)$, and 
\begin{equation}
\forall^{0}\eps\,\forall n\in\N:|a_{n\eps}z_{\eps}^{n}|=|a_{n\eps}\hat{z}_{\eps}^{n}|\cdot\left|\frac{z}{\hat{z}}\right|^{n}<K_{\eps}h_{\eps}^{n}.\label{eq:eventuallybounded-1}
\end{equation}
Thereby, $\sum_{n=N}^{M}|a_{n}z^{n}|\leq\sum_{n=N}^{M}Kh^{n}$ for
all $N$, $M\in\hypNr$. By the direct comparison test (see Thm. \ref{thm:directcomparisonhyperseries}),
the HPS $\hypersum{}{\rho}a_{n}z^{n}$ converges absolutely because
$\hypersum{}{\rho}Kh^{n}$ converges since $h<1$. Moreover, from
\eqref{eq:eventuallybounded-1}, it follows that $\sum_{n=0}^{+\infty}|a_{n\eps}z_{\eps}^{n}|=:R_{\eps}$
converges and is $\rho$-moderate. This implies condition \eqref{eq:eventuallybounded}.

\noindent\ref{enu:uniform}: From \cite[Thm. 74]{Giordano2021},
we have that pointwise convergence implies uniform convergence on
functionally compact sets.

\noindent\ref{enu:interiorofsetofconv}: From the assumptions, $z\in B_{\left|\hat{z}\right|}(0)$
and $\left|\hat{z}\right|<{\rm rad}\left(a_{n}\right)_{{\rm c}}$,
conditions Def.~\ref{def:setofconvergence}\ref{enu:radofconvergence}
and Def.~\ref{def:setofconvergence}\ref{enu:epsHPS} follow. Note
that Def.~\ref{def:setofconvergence}\ref{enu:formalHPS} can be
proved as above from \eqref{eq:eventuallybounded-1}. Finally, if
$\left[\tilde{z}_{\eps}\right]=z$, we have
\begin{align*}
\left|\frac{{\rm d}}{{\rm d}z}\left(\sum_{n=0}^{+\infty}a_{n\eps}z^{n}\right)_{z=\tilde{z}_{\eps}}\right| & =\left|\sum_{n=0}^{+\infty}na_{n\eps}\tilde{z}_{\eps}^{n-1}\right|\\
 & \leq\sum_{n=0}^{+\infty}\left|a_{n\eps}\hat{z}_{\eps}^{n-1}\right|\cdot n\left|\frac{\tilde{z}_{\eps}}{\hat{z}_{\eps}}\right|^{n-1}\\
 & \leq K_{\eps}\left|\hat{z}_{\eps}\right|\sum_{n=0}^{+\infty}n\left|\frac{\tilde{z}_{\eps}}{\hat{z}_{\eps}}\right|^{n-1}\in\R_{\rho},
\end{align*}
where we used Lem.~\ref{lem:boundedsummands}, and hence Def.~\ref{def:setofconvergence}\ref{enu:independence}.
\end{proof}
Based on this theory of generalized complex HPS, we can finally introduce
the notion of generalized complex analytic functions (GCAF) and connect
it with GHF.

\section{\label{sec:gcaf}Generalized complex analytic functions}

For the theory of GHF and their path integrals, we refer to \cite{NuGi24a,NuGi24b}.

A direct consequence of Def.\ref{def:setofconvergence} is the following:
\begin{thm}
\label{thm:analyticthenGHF}Let $[a_{n\eps}]_{{\rm c}}=(a_{n})_{{\rm c}}\in\RCcomplexrho_{{\rm c}}$
(see Def.~\ref{def:radiusofconvergence}) and $c=[c_{\eps}]\in\RCcomplexrho$.
Let $U:=\mathrm{int}\left(\setconv{\rho}{}\left((a_{n})_{{\rm c}},c\right)\right)$
be the sharply interior of the set of convergence. Set 
\[
f(z):=\hypersum{}{\rho}a_{n}(z-c)^{n}=\left[\sum_{n=0}^{+\infty}a_{n\eps}(z_{\eps}-c_{\eps})^{n}\right]=:[v_{\eps}(z_{\eps})],
\]
for all $z=[z_{\eps}]\in U$. Then $f\in\ghf\left(U\right)$ is a
GHF defined by $(v_{\eps})$.
\end{thm}

\begin{proof}
Since $z\in\setconv{\rho}{}\left((a_{n})_{{\rm c}},c\right)$, then
$|z-c|<\text{rad}(a_{n})_{{\rm c}}$, $\left[a_{n\eps}(z_{\eps}-c_{\eps})^{n}\right]_{{\rm s}}\in\RCcomplexrho_{{\rm s}}\left\llbracket z-c\right\rrbracket $,
\[
\hypersum{}{\rho}a_{n}(z-c)^{n}=\left[\sum_{n=0}^{+\infty}a_{n\eps}(z_{\eps}-c_{\eps})^{n}\right]\in\RCcomplexrho,
\]
and for all representatives $[\hat{z}_{\eps}]=z$ we have
\[
\left(\frac{\text{d}}{\text{d}z}\left(\sum_{n=0}^{+\infty}a_{n\eps}(z-c_{\eps})^{n}\right)_{z=\hat{z}_{\eps}}\right)\in\CC_{\rho}.
\]
Therefore, the conclusion directly follows from \cite[Thm.~26]{NuGi24a}.
\end{proof}
Conversely, we conjecture that our GHF are generalized complex analytic
functions, as defined below:
\begin{defn}
Let $U\subseteq\RCcomplexrho$ be a simply connected and sharply open
set. We say that $f:U\ra\RCcomplexrho$ is a \textit{generalized complex
analytic function} (GCAF) on $U$, if for all $c\in U$, we can find
$s\in\RCrealrho_{>0}$, $\left(a_{n}\right)_{{\rm c}}\in\RCcomplexrho_{{\rm c}}$
such that
\begin{enumerate}
\item \label{enu:GCAFsetconv}$B_{s}(c)\subseteq U\cap\setconv{\rho}{}\left((a_{n})_{{\rm c}},c\right)$,
\item \label{enu:analyticatpoint}$f(z)=\sum_{n\in\hypNr}a_{n}\left(z-c\right)^{n}$
for all $z\in B_{s}(c)$.
\end{enumerate}
Moreover, we say that $f:\RCcomplexrho\ra\RCcomplexrho$ is an \emph{generalized
entire function} if we can find $c\in\RCcomplexrho$ and $\left(a_{n}\right)_{{\rm c}}\in\RCrealrho_{{\rm c}}$
such that
\begin{enumerate}[resume]
\item \label{enu:entiresetconv}$\RCcomplexrho=\setconv{\rho}{}\left((a_{n})_{{\rm c}},c\right)$,
\item \label{enu:entire}$f(z)=\sum_{n\in\hypNr}a_{n}\left(z-c\right)^{n}$
for all $z\in\RCcomplexrho$.
\end{enumerate}
We also say that $f$ is \emph{generalized entire at} $c$ if \ref{enu:entiresetconv}
and \ref{enu:entire} hold.
\end{defn}

\noindent Example \ref{exa:diracdelta} shows that Dirac $\delta$
is entire at $0$, but it is not at any $c\in\RCcomplexrho$ such
that $|z|\gg0$.

Clearly, if $\left(a_{n}\right)_{{\rm c}}\in\RCrealrho_{{\rm c}}$,
$c\in\RCcomplexrho$, and we set $f(z):=\sum_{n\in\hypNr}a_{n}\left(z-c\right)^{n}$
then $f$ is generalized complex analytic on the interior of the set
of convergence $\setconv{\rho}{}\left((a_{n})_{{\rm c}},c\right)$,
as proved in Thm.~\ref{thm:analyticthenGHF}. Vice versa, as expected,
our GHF are generalized analytic. For the proof, we follow the classical
strategy, leveraging our generalized framework to bypass the need
for explicit $\eps$-level calculations.
\begin{thm}[Goursat]
\label{thm:goursat} Suppose $U\subseteq\RCcomplexrho$ is a simply
connected and sharply open set. If $f:U\ra\RCcomplexrho$ is a generalized
holomorphic function then $f$ is generalized analytic on $U$.
\end{thm}

\begin{proof}
Let $z_{0}=[z_{0\eps}]\in U$ and $f=\left[f_{\eps}(-)\right]$. Then,
we can choose $R\in\RCrealrho_{>0}$ such that $\overline{B_{R}(z_{0})}\subset U$.
Let $0<s<R$. Set $\gamma_{R}(t):=z_{0}+R\exp(2\pi it)$ for all $t\in[0,1]$,
so that $\gamma_{R}$ is closed, simple, and positively oriented,
and the region bounded by $\gamma_{R}$ is $R_{\gamma_{R}}\subset U$.
Note that for all $z\in B_{s}(z_{0})$, by Cauchy integral formula
we have
\[
f(z)=\frac{1}{2\pi i}\int_{\gamma_{R}}\frac{f(w)}{w-z}\diff w.
\]
For all $w\in\gamma_{R}$, we have $R=|w-z_{0}|>0$, $\left|\frac{z-z_{0}}{w-z_{0}}\right|<\frac{s}{R}<1$
and hence the geometric hyperseries $\hypersum{\rho}{}\left(\frac{z-z_{0}}{w-z_{0}}\right)^{n}=\frac{1}{1-\left|\frac{z-z_{0}}{w-z_{0}}\right|}$
and 
\begin{align*}
\frac{1}{w-z} & =\frac{1}{w-z_{0}-(z-z_{0})}=\frac{1}{w-z_{0}}\left(\frac{1}{1-\left|\frac{z-z_{0}}{w-z_{0}}\right|}\right)\\
 & =\frac{1}{w-z_{0}}\left(\phantom{|}\hypersum{}{\rho}\left(\frac{z-z_{0}}{w-z_{0}}\right)^{n}\right).
\end{align*}
So, we have
\[
\frac{f(w)}{w-z}=\frac{f(w)}{w-z_{0}}\left(\phantom{|}\hypersum{}{\rho}\left(\frac{z-z_{0}}{w-z_{0}}\right)^{n}\right)=\hypersum{}{\rho}f(w)\frac{\left(z-z_{0}\right)^{n}}{\left(w-z_{0}\right)^{n+1}}
\]
and according to convergence theorem on functionally compact sets
\cite[Thm. 75]{Giordano2021}, pointwise convergence implies uniform
convergence on functionally compact set, and hence, we can exchange
integration and hyperseries to have
\begin{align*}
f(z) & =\frac{1}{2\pi i}\int_{\gamma_{R}}\hypersum{}{\rho}f(w)\frac{\left(z-z_{0}\right)^{n}}{\left(w-z_{0}\right)^{n+1}}\,\diff w\\
 & =\hypersum{}{\rho}\left(\frac{1}{2\pi i}\int_{\gamma_{R}}\frac{f(w)}{\left(w-z_{0}\right)^{n+1}}\,\diff w\right)\left(z-z_{0}\right)^{n}.
\end{align*}
\end{proof}
Classical properties of the interplay between holomorphic and analytic
functions are Liouville's theorem, the identity principle, and Paley-Wiener
theorem. We first prove Liouville's theorem:
\begin{thm}[Liouville]
\label{thm:Liouville}Every bounded entire $f\in\ghf(\RCcomplexrho)$,
i.e.~for which there exists $M\in\RCrealrho_{>0}$ such that $|f(z)|<M$
for all $z\in\RCcomplexrho$, is constant.
\end{thm}

\begin{proof}
Without loss of generality, we can assume that $f$ is entire at $c=0$.
By Goursat's theorem \ref{thm:goursat}, for all $z\in\RCcomplexrho$
we have
\[
f(z)=\hypersum{}{\rho}a_{n}z^{n}\;\text{where}\;a_{n}=\frac{f^{(n)}(0)}{n!}=\frac{1}{2\pi i}\int_{\gamma_{r}}\frac{f(w)}{w^{n+1}}\,\diff w
\]
where $\gamma_{r}$ is any circle centered at $0$ of radius $r\in\RCrealrho_{>0}$
since $f$ is entire in $\RCcomplexrho$. Therefore,
\begin{align*}
0\le|a_{n}|=\left|\frac{1}{2\pi i}\int_{\gamma_{r}}\frac{f(w)}{w^{n+1}}\diff w\right| & \leq\frac{1}{2\pi}\int_{\gamma_{r}}\frac{|f(w)|}{|w|^{n+1}}\,\diff w\\
 & \leq\frac{1}{2\pi}\int_{\gamma_{r}}\frac{M}{r^{n+1}}\diff w\\
 & =\frac{M}{2\pi r^{n+1}}2\pi r\\
 & =\frac{M}{r^{n}}.
\end{align*}
Letting $r\to+\infty$ in the sharp topology gives $|a_{n}|=0$ for
all $n\in\N$. Thus, $f(z)=a_{0}$.
\end{proof}
Before we study the identity theorem for generalized holomorphic functions,
we extend one important properties of a holomorphic function i.e.~its
zeroes are isolated (assuming the function is not identically zero):
\begin{thm}
\label{thm:zeros}Let $U\subseteq\mathbb{\RCcomplexrho}$ be a sharply
open set, $z_{0}\in U$, and $f\in\ghf(U)$. If 
\begin{equation}
f(z_{0})=0,\;\exists n\in\N_{>0}:f^{(n)}(z_{0})\neq0,\label{eq:zeros}
\end{equation}
then there exist $L=L(z_{0})\subseteq_{0}I$ and $r\in\RCrealrho_{>0}$
such that 
\begin{equation}
\forall z\in B_{r}(z_{0}):|z-z_{0}|>0\implies f(z)\mid_{L}\;\text{is invertible.}\label{eq:zeros-1}
\end{equation}
\end{thm}

\begin{proof}
Since $f\in\ghf(U)$, we have
\begin{equation}
\exists s\in\RCrealrho_{>0}\,\forall z\in B_{s}(z_{0}):f(z)=\hypersum{}{\rho}\frac{f^{(n)}(z_{0})}{n!}(z-z_{0})^{n}.\label{eq:analytic}
\end{equation}
From \eqref{eq:zeros}, there exists $m:=\min\left\{ n\in\N_{>0}:f^{(n)}(z_{0})\neq0\right\} >0$,
and $f^{(m)}(z_{0})\neq0$. On the other hand, from \eqref{eq:analytic},
we have
\[
f(z)=(z-z_{0})^{m}\cdot\sum_{n\in\hypNr_{\geq m}}\frac{f^{(n)}(z_{0})}{n!}(z-z_{0})^{n-m}=:(z-z_{0})^{m}\cdot g(z),
\]
with $g(z_{0})=f^{(m)}(z_{0})\neq0$. This means $\exists L\subseteq_{0}I:g(z_{0})\mid_{L}$
is invertible, hence $\left|g(z_{0})\right|\mid_{L}>0$. Since $g(-):B_{s}(z_{0})\ra\RCcomplexrho$
is sharply continuous, then $\exists r\in(0,s)$ $\forall z\in B_{r}(z_{0}):\left|g(z)\right|\mid_{L}>0$.
Therefore:
\[
\forall z\in B_{r}(z_{0}):|z-z_{0}|>0\implies|f(z)|=|z-z_{0}|\cdot|g(z)|,
\]
which proves our claim: $f(z)\mid_{L}$ is invertible.
\end{proof}
Note that \eqref{eq:zeros} represents the property of being non-zero.
Moreover, the conclusion \eqref{eq:zeros-1} states that $z_{0}$
is isolated, at least on a suitable subpoint.

As we already discussed in \cite[Sec. 4]{NuGi24a}, since $\delta$
is a GHF, and hence is a generalized analytic function, in general
the identity principle does not hold for GHF. The following theorem
shows that the identity principle does not hold in our framework because
of two motivations: we are in a non-Archimedean setting and hence
the set of all the infinitesimals is a clopen set (see \cite[Thm.~43]{NuGi24a}),
and because analytic continuation from a point $z_{0}$ to another
$z$ cannot, in general, be reversed for GHF:
\begin{thm}[Identity principle]
\label{thm:identity}Let $U\subseteq\mathbb{\RCcomplexrho}$ be a
sharply open set, $(z_{k})_{k\in\hypNr}$ be a hypersequence of points
in $U$ such that $\hyplimsarg{\rho}{k}z_{k}=z_{0}\in U$, $\left|z_{k}-z_{0}\right|>0$
for all $k\in\hypNr$, and $f\in\ghf(U)$ such that $f(z_{k})=0$
for all $k\in\hypNr$. Assume that $f$ \textup{can be continued}
$z_{0}\ra z$ (we read it: ``from $z_{0}$ to z'') for all $z\in U$,
i.e.~there exist $N\in\N$, $c_{1},c_{2},\dots,c_{N}\in U$ and $r_{1},r_{2},\dots,r_{N}\in\RCrealrho_{>0}$
such that:
\begin{enumerate}
\item \label{enu:identity1}$c_{1}=z_{0}$, $c_{N}=z$,
\item \label{enu:identity2}$c_{k+1}\in B_{r_{k}}(c_{k})$ for all $k=1,\dots,N-1$,
\item \label{enu:identity3}$B_{r_{k}}(c_{k})\subseteq\setconv{\rho}{}\left(f,c_{k}\right)$
for all $k=1,\dots,N$.
\end{enumerate}
Then,
\[
f(z)=0\quad\forall z\in U.
\]
\end{thm}

\begin{proof}
We consider two cases:

\noindent Case I: If $U\subseteq\setconv{\rho}{}\left(f,z_{0}\right)$,
then 
\[
f(z)=\hypersum{}{\rho}\frac{f^{(n)}(z_{0})}{n!}(z-z_{0})^{n}\quad\forall z\in U.
\]
By contradiction, assume $\exists n\in\N_{>0}:f^{(n)}(z_{0})\neq0$.
As we did in the proof of Thm.~\ref{thm:zeros}, there exists $m:=\min\left\{ n\in\N_{>0}:f^{(n)}(z_{0})\neq0\right\} >0$,
which implies $f^{(m)}(z_{0})\neq0$, and there exist $L=L(z_{0})\subseteq_{0}I$,
a continuous function $g$, and $r\in\RCrealrho_{>0}$ such that 
\begin{equation}
\forall z\in B_{r}(z_{0}):|z-z_{0}|>0\implies\left|f(z)\right|\mid_{L}=|z-z_{0}|\cdot|g(z)|\mid_{L}>0\label{eq:case1}
\end{equation}
with $g(z_{0})=f^{(m)}(z_{0})\neq0$. However, the assumption $\hyplimsarg{\rho}{k}z_{k}=z_{0}\in U$,
yields the existence of $K\in\hypNr$ such that $z_{K}\in B_{r}(z_{0})$
and $|z_{K}-z_{0}|>0$. This means $f(z_{K})=0$, which implies $g(z_{K})=0$
and hence also $|g(z_{K})|\mid_{L}=0$, which is a contradiction with
\eqref{eq:case1}.

\noindent Case II: Let $z\in U$ and only in this case assume that
$f$ can be extended $z_{0}\ra z$. From Case I, we have $f\mid_{B_{r_{0}}(c_{0})}=0$
becausee $B_{r_{0}}(c_{0})\subseteq\setconv{\rho}{}\left(f,c_{0}\right)$.
Since $c_{1}\in B_{r_{0}}(c_{0})$, there exists $(z_{1k})_{k\in\hypNr}$
sequence of $B_{r_{0}}(c_{0})$ converging to $c_{1}$ and such that
$\left|z_{1k}-c_{1}\right|>0$. Since $\hyplimsarg{\rho}{k}z_{1k}=c_{1}$,
$f(z_{1k})=0$ for all $k\in\N$, we have $f\mid_{B_{r_{1}}(c_{1})}=0$.
Since we only have a finite number of $c_{k}$, by induction, we have
$f\mid_{B_{r_{N}}(c_{N})}=0$ and this proves our claim because $c_{N}=z$.
\end{proof}
\begin{rem}
~
\begin{enumerate}
\item The Dirac $\delta$ can be extended $0\ra1$, but it cannot be $1\ra0$.
Note that $\delta\mid_{B_{r}(0)}\neq0$ whereas $\delta\mid_{B_{s}(1)}=0$,
so \ref{enu:identity2} cannot hold. Similarly, the Heaviside function
on $\RCcomplexrho$, see \cite{NuGi24b}, cannot be continued $1\ra0$
but it can be continued $0\ra1$.
\item The continuation assumption is invoked exclusively in Case II. Since
we used induction, it is important that in this case $N\in\N$ is
a standard natural number.
\end{enumerate}
\end{rem}

\section{A Paley-Wiener type theorem for the hyperfinite Fourier transform}

In this last section, we generalize a Paley-Wiener type theorem to
the framework of GSF and GHF. The classical Paley-Wiener theorems
make use of the holomorphic Fourier transform on classes of square-integrable
functions supported on the real line, and relates decay properties
of a function or distribution at infinity (such as being compactly
supported) with analyticity of its Fourier transform (FT). On the
contrary to the classical setting, we will see that the hyperfinite
Fourier transform (HFT) of \emph{any} GSF $f$ is always a GHF, without
assuming any compactness support of $f$. Moreover, the HFT is always
an entire GHF if the domain of integration is of logarithmic type.

The basic idea to define a very general FT for GSF is the following,
see \cite{MTG}: Since GSF are based on a non-Archimedean ring of
scalars, we can consider a positive \emph{infinite} generalized number
$h$ and define the FT with the usual formula, but integrating over
the interval $[-h,h]$. We know that, although $h$ is an infinite
number (hence, $[-h,h]\supseteq\R$!), this interval behaves like
a compact set for GSF, so that, e.g., on these integration domains
we always have an extreme value theorem and integrals always exist.
Clearly, this leads to a FT, called \emph{hyperfinite} FT, that depends
on the parameter $h$, but, on the other hand, where we can transform
\emph{all} the GSF defined on this interval, and these include all
tempered Schwartz distributions, all tempered Colombeau GF, but also
a large class of non-tempered GF, such as the exponential functions,
or non-linear examples like $\delta^{a}\circ\delta^{b}$, $\delta^{a}\circ H^{b}$,
$a$, $b\in\N$, etc. Not all the properties of the classical FT remain
unchanged for this more general transform, but the final formalism
still retains the useful properties of the FT in dealing with differential
equations. Even more, the new formula for the transform of derivatives
leads us to include also exponential solutions in the Fourier method
applied to the simplest ODE: $y'=y$, $y(0)=1$, see \cite{MTG}.

We start by first recalling the notion of compactly supported GSF,
because we will use it to fully develop the example of the HFT of
Dirac delta. In the generalized functions framework, compactly supported
GSF were introduced in \cite{MTG} for an arbitrary gauge, and in
\cite{Giordano2018} for the gauge $\rho_{\eps}=\eps$.
\begin{defn}
\label{def:compactlysupportedGSF}Assume that $X\subseteq\RCrealrho$
and $f\in\gsf\left(X,\RCcomplexrho\right)$, then
\begin{enumerate}
\item \label{enu:support}${\rm supp}\left(f\right):=\overline{\left\{ x\in X\mid\left|f(x)\right|>0\right\} }$,
where $\overline{\left(\cdot\right)}$ denotes the relative closure
in $X$ with respect to the sharp topology, is called \textit{the
support of} $f$. Note that ${\rm supp}\left(f\right)$ is clearly
always sharply closed. However, in general it is not an internal set,
so it is not a functionally compact set.
\item \label{enu:exterior}For $A\subseteq\RCrealrho$, we call the set
${\rm ext}\left(A\right):=\left\{ x\in\RCrealrho\mid\forall a\in A:\left|x-a\right|>0\right\} $
the strong exterior of $A$. For example, using the generalized Jordan
curve theorem \cite[Thm.~40]{NuGi24b}, it is easy to prove that if
$\gamma\in\gsfk{1}([a,b])$ is closed, simple and positively oriented
path, then ${\rm ext}\left(\overline{I}_{\gamma}\right)=E_{\gamma}$.
\item \label{enu:compactlysupportedfunction}Let $H\Subset_{{\rm f}}\RCrealrho$
be a functionally compact set, we say that $f$ is \emph{compactly
supported in} $H$, and write $f\in\Dgsf\left(H,\RCcomplexrho\right)$,
if $f\in\gsf\left(\RCrealrho,\RCcomplexrho\right)$ and ${\rm supp}(f)\subseteq H$.
We say that $f\in\Dgsf\left(\RCrealrho,\RCcomplexrho\right)$, if
${\rm supp}\left(f\right)\subseteq H$ for some functionally compact
set $H\Subset_{{\rm f}}\RCrealrho$. Since $H\Subset_{{\rm f}}\RCrealrho$,
then $f$ is also sharply bounded by the extreme value Thm~6 of \cite{NuGi24b}.
\end{enumerate}
\end{defn}

\begin{thm}
\label{thm:supportexterior}Let $f\in\gsf\left(\RCrealrho,\RCcomplexrho\right)$
and $\emptyset\neq H\Subset_{{\rm f}}\RCrealrho$. Then,
\begin{enumerate}
\item \label{enu:extopen}${\rm ext}(H)$ is sharply open, and
\item \label{enu:supportiffext}$f\in\Dgsf\left(H,\RCcomplexrho\right)$
if and only if $f\mid_{{\rm ext}(H)}=0$.
\item \label{enu:derivative}If $f\in\Dgsf\left(H,\RCcomplexrho\right)$
then $f^{(n)}(x)=0$ for all $x\in{\rm ext}(H)$ and $n\in\N$. In
particular, if $H\subseteq[-h,h]$, $h>0$, then $f^{(n)}(x)=0$ whenever
$x\in(-\infty,-h]\vee[h,+\infty)$ (see \cite[Def.~35]{NuGi24b} for
the definition of $\vee$).
\end{enumerate}
\end{thm}

\begin{proof}
See \cite[Lem.~3.3]{MTG} and \cite[Thm.~3.4]{MTG}.
\end{proof}
\begin{example}
~
\begin{enumerate}
\item Since the $1$-dimensional Dirac delta $\delta_{1}\in\gsf\left(\RCrealrho,\RCcomplexrho\right)$
and $\delta_{1}(0)>0$, then $\delta_{1}\mid_{B_{s}(0)}>0$ for some
$s\in\RCrealrho_{>0}$ by the sharp continuity of $\delta_{1}$. Hence
$B_{s}(0)\subseteq{\rm supp}\left(\delta_{1}\right)\subseteq[-r,r]$,
for all $r\in\R_{>0}$.
\item Consider the GSF corresponding to the Gaussian $f(x):=e^{-\frac{|x|^{2}}{2}}$,
for all $x\in\RCrealrho$. Then $f\in\gsf\left(\RCrealrho,\RCcomplexrho\right)$
and it is not hard to prove that $0\leq f(x)\leq\left|x\right|^{-q}$
for all $q\in\N$ and for $|x|$ finite sufficiently large. Therefore,
for any strongly infinite numbers $h\in\RCrealrho_{>0}$, we have
$f(x)=0$ for all $x\in{\rm ext}\left(\left[-h,h\right]\right)$.
Hence, ${\rm supp}(f)\subseteq[-h,h]$.
\end{enumerate}
\end{example}

We now define the hyperfinite Fourier transform (HFT). For more details
and proofs about HFT, the reader can refer to \cite{MTG}.
\begin{defn}
\label{def:Fouriertransform}Let $h\in\RCrealrho_{>0}$ be a positive
infinite number, $H:=[-h,h]\subseteq\RCrealrho$, and $f\in\gsf(H,\RCcomplexrho)$.
We define \textit{the hyperfinite Fourier transform (HFT) $\mathcal{F}_{h}(f)$
of $f$ on} $H$ as follows: 
\begin{equation}
\mathcal{F}_{h}(f)(\omega):=\int_{-h}^{h}f(x)e^{-ix\omega}\,\diff x,\quad\forall\omega\in\RCcomplexrho.\label{eq:HFT-1}
\end{equation}
\end{defn}

The term \textit{hyperfinite} captures a deliberate duality: while
$h\in\RCrealrho$ is an infinite number, generalized smooth functions
treat functionally compact sets like $H$ as though they were finite
with compact structures. Thus, \textit{hyperfinite} reflects an entity
that is infinite in scale but bounded in several points of view.

Clearly, $\mathcal{F}_{h}:\gsf(H,\RCcomplexrho)\ra\gsf(\RCcomplexrho,\RCcomplexrho)$
and for all $f\in\gsf(H,\RCcomplexrho)$ and $\omega\in\RCrealrho$
we have 
\begin{equation}
\left|\mathcal{F}_{h}(f)(\omega)\right|\leq\int_{H}|f(x)|\,\diff x,\label{eq:HFT}
\end{equation}
so that the HFT is always sharply bounded on $\RCrealrho$. We also
note that if $h=[h_{\eps}]$, then
\[
\mathcal{F}_{h}\left(f\right)\left(\omega\right)=\left[\intop_{-h_{\eps}}^{h_{\eps}}f_{\eps}\left(x\right)e^{-ix\cdotp\omega_{\eps}}\,\diff{x}\right]=\left[\hat{\mathcal{F}}(\chi_{H_{\eps}}f_{\eps})(\omega_{\eps})\right]\in\RCcomplexrho,
\]
where $\hat{\mathcal{F}}:\mathcal{S}(\R)\ra\mathcal{S}(\R)$ is the
classical FT, and $\chi_{H_{\eps}}$ is the characteristic function
of $H_{\eps}:=[-h_{\eps},h_{\eps}]$.

We naturally define the inverse HFT as follows:
\begin{defn}
\label{def:inverseFT}Let $h\in\RCrealrho_{>0}$ be a positive infinite
number, $H:=[-h,h]\Subset_{{\rm f}}\RCrealrho$, and $f\in\gsf(H,\RCcomplexrho)$.
The inverse HFT is 
\[
\mathcal{F}_{h}^{-1}(f)(x):=\frac{1}{2\pi}\int_{-h}^{h}f(\omega)e^{ix\omega}\,\diff\omega\quad\forall x\in\RCcomplexrho.
\]
\end{defn}

Similarly, $\mathcal{F}_{h}^{-1}:\gsf(H,\RCcomplexrho)\ra\gsf(\RCcomplexrho,\RCcomplexrho)$.
We should clarify that the notation $\mathcal{F}_{h}^{-1}$ is an
abuse of notation, as the codomain of $\mathcal{F}_{h}$ does not
coincide with the domain of $\mathcal{F}_{h}^{-1}$ (and vice-versa).
When dealing with inversion properties, it is therefore preferable
to consider that 
\begin{align*}
\mathcal{F}_{h}\mid_{H} & :=\left(-\right)\mid_{H}\circ\mathcal{F}_{h}:\gsf(H,\RCcomplexrho)\ra\gsf(H,\RCcomplexrho)\\
\mathcal{F}_{h}^{-1}\mid_{H} & :=\left(-\right)\mid_{H}\circ\mathcal{F}_{h}^{-1}:\gsf(H,\RCcomplexrho)\ra\gsf(H,\RCcomplexrho).
\end{align*}

For this HFT, several properties can be extended, see \cite{MTG}.
The main differences concern the HFT of a derivative
\begin{align}
\mathcal{F}_{h}\left(f'\right) & =i\omega\mathcal{F}_{h}\left(f\right)+\Delta_{1h}f\label{eq:HFTDer}\\
\Delta_{1h}f(\omega): & =\left[f(x)e^{-ix\cdot\omega}\right]_{x=-h}^{x=h}.\nonumber 
\end{align}
and the Fourier inversion theorem
\[
\lim_{k\to+\infty}\mathcal{F}_{k}^{-1}\left(\mathcal{F}_{h}(f)\right)(y)=\lim_{k\to+\infty}\mathcal{F}_{k}\left(\mathcal{F}_{h}^{-1}(f)\right)(y)=f(y)\quad\forall y\in\mathring{H}.
\]
We also finally recall that the $\Delta_{1h}f$ term appearing in
\eqref{eq:HFTDer} is essential e.g.~to obtain non trivial solutions
of differential equations, such as the exponential one of the simplest
ODE: $y'=y$, $y(0)=1$ (this example also emphasizes the importance
to consider the HFT $\mathcal{F}_{h}(y)$ for any GSF $y$, not necessarily
of tempered type).

The following result represents the Riemann-Lebesgue lemma in our
framework. It immediately highlights an important difference with
respect to the classical approach since it states that the restriction
to $\RCrealrho$ of the HFT of a very large class of compactly supported
GSF is still compactly supported.
\begin{lem}[Riemann-Lebesgue]
\label{lem:RiemannLebesgue}Let $H=[H_{\eps}]\Subset_{{\rm f}}\RCrealrho$
and $f\in\Dgsf\left(H,\RCcomplexrho\right)$ be a compactly supported
GSF. Assume that $f$ is uniformly bounded by a tame polynomial, i.e.\ref{eq:tamepolynom}
\begin{equation}
\exists C,b\in\RCrealrho_{>0}\,\forall x\in H\,\forall j\in\N:\left|f^{(j)}(x)\right|\leq C\cdot b^{j}.\label{eq:tamepolynom}
\end{equation}
For all $N\in\N$ and $\omega\in\RCrealrho$, if $\omega$ is invertible,
then
\begin{equation}
\left|\mathcal{F}_{h}(f)(\omega)\right|\leq\frac{1}{|\omega|^{N}}\cdot\int_{H}\left|f^{(N)}(x)\right|\,\diff x.\label{eq:HFTbounded}
\end{equation}
Therefore, $\lim_{\omega\rightarrow+\infty}\left|\mathcal{F}_{h}(f)(\omega)\right|=0$
(in the sharp topology). Actually, \eqref{eq:HFTbounded} yields the
stronger result:
\begin{equation}
\exists Q\in\N:\mathcal{F}_{h}(f)|_{\RCrealrho}\in\Dgsf\left(\overline{B_{\diff\rho^{-Q}}(0)},\RCcomplexrho\right).\label{eq:supportHFT}
\end{equation}
\end{lem}

\begin{proof}
For the first part of the proof, see \cite[Lem.~5.1]{MTG}. Here,
we only prove \eqref{eq:supportHFT}. We first recall that $\overline{B_{\diff\rho^{-Q}}(0)}\Subset_{{\rm f}}\RCrealrho$.
Let $C$, $b\in\RCrealrho_{>0}$ from \eqref{eq:tamepolynom} and
\[
\lambda(H):=\lim_{\substack{m\rightarrow+\infty\\
m\in\N
}
}\left[\lambda\left(\overline{\Eball}_{\rho_{\eps}^{m}}\left(H_{\eps}\right)\right)\right]\in\RCrealrho,
\]
where $\lambda$ is the Lebesgue measure. Therefore, $b\leq\diff\rho^{-R}$
for some $R\in\N$, and we can set $Q:=R+1$. We want to prove the
claim using Thm. \ref{thm:supportexterior}\ref{enu:supportiffext},
so that we take $\omega\in{\rm ext}\left(\overline{B_{\diff\rho^{-Q}}(0)}\right)$.
It cannot be $|\omega|<_{{\rm s}}\diff\rho^{-Q}$ because this would
yields $\left|\omega-a\right|=_{{\rm s}}0$ for some $a\in\overline{B_{\diff\rho^{-Q}}(0)}$;
consequently, $|\omega|\geq\diff\rho^{-Q}$. We apply assumption \eqref{eq:tamepolynom}
and inequality \eqref{eq:HFTbounded} with an arbitrary $N\in\N$
to get
\begin{align*}
\left|\mathcal{F}_{h}(f)(\omega)\right| & \leq\frac{1}{|\omega|^{N}}\cdot\int_{H}\left|f^{(N)}(x)\right|\,\diff x\\
 & \leq\diff\rho^{NQ}Cb^{N}\lambda(H)\\
 & \leq C\diff\rho^{N(Q-R)}\lambda(H)=C\diff\rho^{N}\lambda(H).
\end{align*}
For $N\rightarrow+\infty$, we hence have $\mathcal{F}_{h}(f)(\omega)=0$
if $\omega\in\RCrealrho$.
\end{proof}
We close this section with a Paley-Wiener type theorem. We again underscore
that this result states that the HFT $\mathcal{F}_{h}(f)$ of \emph{any}
GSF $f\in\gsf(H,\RCcomplexrho)$ is always an entire GHF without assuming
that $f$ is compactly supported. As it will be clear from the proof,
this is a consequence of the functionally compact integration domain
$H$, even if $H\supseteq\R$. In the end, it is a consequence of
having a good non-Archimedean setting with infinite numbers.
\begin{thm}[Paley-Wiener]
\label{thm:paleywiener}Let $h\in\RCrealrho_{>0}$, set $H:=[-h,h]$,
and let $f\in\gsf(H,\RCcomplexrho)$. Then the HFT $\mathcal{F}_{h}(f)$
satisfies the following properties:
\begin{enumerate}
\item \label{enu:HFT_GHF}$\mathcal{F}_{h}(f)\in\ghf(\RCcomplexrho,\RCcomplexrho)$
is a GHF.
\item \label{enu:HFTEntire}If $h$ is of logarithmic type, i.e.~$h\le Q\log\diff\rho$,
for some $Q\in\R_{>0}$, then $\mathcal{F}_{h}(f)$ is an entire GHF.
\item \label{enu:HFTIntegrabile}It is integrabile on the real axis: $\int_{H}\left|\mathcal{F}_{h}(f)\mid_{\RCrealrho}(\omega)\right|^{2}\diff\omega=2\pi\int_{H}\left|f\right|^{2}\in\RCrealrho$
(Plancherel).
\item \label{enu:HFTExp}It is of exponential type: There exists a constant
$C\in\RCrealrho_{>0}$ and for all $z\in\RCcomplexrho$:
\[
\left|\mathcal{F}_{h}(f)(z)\right|\leq Ce^{h|z|}.
\]
\end{enumerate}
\end{thm}

\begin{proof}
Set $\gamma_{s}$ be the circle of radius $s\in\RCrealrho_{>0}$ and
centered at $z\in\RCcomplexrho$. By Fubini theorem, we have
\begin{align*}
\int_{\gamma_{r}}\mathcal{F}_{h}(f)(z)\,\diff z & =\int_{\gamma_{r}}\left(\int_{H}f(\omega)e^{-iz\omega}\,\diff\omega\right)\,\diff z\\
 & =\int_{H}f(\omega)\left(\int_{\gamma_{r}}e^{-iz\omega}\,\diff\omega\right)\,\diff z.
\end{align*}
Since $e^{-iz\omega}$ is a generalized entire function of $z$, the
inner integral $\int_{\gamma_{r}}e^{-iz\omega}\,\diff\omega=0$ by
Cauchy-Goursat's theorem (see \cite[Thm. 45]{NuGi24b}). Therefore,
the whole expression is $0$, and by Morera's theorem (see \cite[Thm. 41]{NuGi24b}),
$\mathcal{F}_{h}(f)$ is a GHF. Now, using the Taylor hyperseries
of $\omega\mapsto e^{-iz\omega}$, see Example \ref{exa:exponential},
since $h$ is of logarithmic type and $\omega\in H$, we have
\begin{align*}
\mathcal{F}_{h}(f)(\omega) & =\int_{H}f(\omega)\hypersum{}{\rho}\frac{(-iz\omega)^{n}}{n!}\,\diff\omega\\
 & =\hypersum{}{\rho}\frac{(-iz)^{n}}{n!}\int_{H}f(\omega)\omega^{n}\,\diff\omega,
\end{align*}
where we use \cite[Thm.~75]{Giordano2021} to exchange hyperseries
and integration over the functionally compact set $H$. This proves
that $\mathcal{F}_{h}(f)$ is an entire function.

By Plancherel's identity (see \cite[Thm. 7.11]{MTG}), we have: 
\begin{align*}
\int_{H}\left|\mathcal{F}_{h}(f)\mid_{\RCrealrho}(\omega)\right|^{2}\diff\omega & =2\pi\int_{H}\left|f\right|^{2}\in\RCrealrho.
\end{align*}
Moreover, let $z=x+iy$, then 
\begin{align*}
\left|\mathcal{F}_{h}(f)(z)\right| & =\left|\int_{H}f(\omega)e^{-iz\omega}\,\diff\omega\right|\\
 & \leq\int_{H}\left|f(\omega)\right|\left|e^{-iz\omega}\right|\,\diff\omega\\
 & =\int_{H}\left|f(\omega)\right|e^{y\omega}\,\diff\omega.
\end{align*}
We then apply the Cauchy-Schwarz inequality, which in our setting
relies upon the monotonicity of the integral, Fubini’s theorem, and
Young’s inequality. Since the classical proofs of these underlying
properties can be directly adapted to our framework, the inequality
holds as expected (see \cite[Thm. 3.10]{MTG}). We have:
\begin{align*}
\left|\mathcal{F}_{h}(f)(z)\right| & \leq\int_{H}\left|f(\omega)\right|e^{y\omega}\,\diff\omega\\
 & \leq\left(\int_{H}\left|f(\omega)\right|^{2}\,\diff\omega\right)^{\nicefrac{1}{2}}\left(\int_{H}e^{2h|y|}\,\diff\omega\right)^{\nicefrac{1}{2}}\\
 & \leq e^{h|z|}\left(\int_{H}\left|f(\omega)\right|^{2}\,\diff\omega\right)^{\nicefrac{1}{2}}\left(\int_{H}\diff\omega\right)^{\nicefrac{1}{2}}\\
 & =Ce^{h|z|}.
\end{align*}
\end{proof}
For example, consider the $1$-dimensional Dirac delta $\delta_{1}(x):=\diff\rho^{-1}\mu(\diff\rho x)$
where $\mu:=\mathcal{F}^{-1}(\beta)$ is an entire mollifier defined
as the inverse Fourier transform of any smooth compactly supported
function $\beta\in\mathcal{D}(\R)$ with $\beta(0)=1$ (see \cite[Sec. 5]{NuGi24a}
for more details). We have $\delta_{1}\in\Dgsf(H,\RCcomplexrho)$.
For all $x\in\RCrealrho$, we have $\delta_{1}^{(j)}(x)=\diff\rho^{-j-1}\mu^{(j)}(\diff\rho^{-1}x)$
and
\begin{align*}
\delta_{1}^{(j)}(x) & =\diff\rho^{-j-1}\left[\frac{{\rm d}^{j}}{{\rm d}x^{j}}\left(\frac{1}{2\pi}\int\beta(t)e^{i\rho_{\eps}^{-1}xt}\,\diff t\right)\right]\\
 & =\frac{\diff\rho^{-j-1}}{2\pi}\left[\int\left(i\rho_{\eps}^{-1}t\right)^{j}\beta(t)e^{i\rho_{\eps}^{-1}xt}\,\diff t\right]\\
\left|\delta_{1}^{(j)}(x)\right| & \leq\frac{\diff\rho^{-2j-1}}{2\pi}\int_{-1}^{1}|t|^{j}\beta(t)\,\diff t=:C\left(\diff\rho^{-2}\right)^{j}.
\end{align*}
Thus, the $1$-dimensional Dirac's delta satisfies condition \eqref{eq:tamepolynom},
i.e.~it is uniformly bounded by a tame polynomial. By Riemann-Lebesgue
Lem.~\ref{lem:RiemannLebesgue}, we have
\[
\exists Q\in\N:f=\mathcal{F}_{h}(\delta_{1})|_{\RCrealrho}\in\Dgsf\left(\overline{B_{\diff\rho^{-Q}}(0)}\right).
\]
On the other hand, the Paley-Wiener Thm.~\ref{thm:paleywiener},
yields that $\mathds{1}:=\mathcal{F}_{h}(\delta_{1})$ is an entire
GHF if $h$ is of logarithmic type, hence it is surely \emph{not}
compactly supported because of Liouville Thm.~\ref{thm:Liouville}.
It is possible to prove that $\mathds{1}(x)=1$ for all \emph{finite}
$x\in\RCrealrho$, see \cite[Rem.~5.2]{MTG}. Moreover, it is also
possible to prove that if $f\in\mathcal{S}(\R)$ and $T\in\mathcal{S}'(\R)$,
then $\mathcal{F}(f)=\mathcal{F}(f\cdot\mathds{1})=\iota(\hat{f})$
and $\mathcal{F}_{h}(\iota(T))=\iota(\hat{T})$, where $\hat{(-)}$
is the classical FT and $\iota$ is the embedding of distributions
(see \cite{NuGi24a}).

The converse of Thm.~\ref{thm:paleywiener} remains an open problem.

\section{Comparison with Colombeau theory and conclusions}

The analyticity of Colombeau holomorphic generalized functions (see
e.g. \cite{Aragona2005,Aragona2012,Colombeau1984,Colombeau1984b,Oberguggenberger2007,Vernaeve2008,Juriaans2020}and
references therein) has been studied at different level of generality.
For instance, \cite{Colombeau1984,Colombeau1984b} note that because
holomorphic generalized functions are defined on $\CC$, classical
Taylor series do not allow one to obtain several desired results.
Specifically, \cite[Sec. 7]{Colombeau1984} provides three examples
of holomorphic generalized functions that fail to satisfy the principle
of analytic continuation.

Conversely, a version of Goursat's theorem was proven in \cite{Aragona2012}
as a consequence of the theory of path integration for generalized
functions on generalized subsets, known in \cite{Aragona2012} as
membranes (notably, a membrane is an internal set). This work was
preceded by the theory of discontinuous Colombeau differential calculus
\cite{Aragona2005}, which allows for the arbitrary composition of
suitable classes of generalized functions. However, in \cite{Aragona2012},
path integral are considered only for the composition $f(\gamma_{t})$
with a Colombeau generalized function $f\in\mathcal{G}(\Omega)$,
$\Omega\subseteq\CC$. This composition is defined only if the path
is generated by a history $(\gamma_{\eps})$, so that $(\gamma_{\eps}([0,1]))$
is a pre-membrane, a condition similar to c-boundedness (see \cite[Def.~3.1, 3.2]{Aragona2012}).
Consequently, Goursat's theorem holds only on the set of compactly
supported points, which, we recall, can be defined as: 
\[
\left\langle \Omega\right\rangle _{{\rm fin}}:=\left\{ z\in\left\langle \Omega\right\rangle \mid z\text{ is finite},\exists r\in\R_{>0}:d\left(z,\partial\Omega\right)\geq r\right\} 
\]
(see \cite[Thm. 5.2]{Aragona2005}).

Subsequently, Vernaeve \cite{Vernaeve2008} used an approach equivalent
to GSF theory, but using quotients instead of directly treating generalized
functions as ordinary maps. As established in \cite[Prop. 4.18]{Vernaeve2008}
(recall also Sec.~\ref{sec:Introduction}), while each complex analytic
Colombeau generalized function can be expressed as a Taylor series,
this expansion is valid only within an infinitesimal neighborhood
of each point.

In general, we can conclude this article by considering that in the
study of PDE, the class of analytic functions is sometimes described
as a too rigid set of solutions. In spite of their good properties
with respect to algebraic operations, composition, differentiation,
integration, inversion, etc., this rigidity is essentially well represented
by the identity principle that necessarily excludes e.g.~solitons
with compact support or interesting generalized functions. Thanks
to the identity principle \ref{thm:identity}, we can state that this
rigidity stems from two main factors: the exclusion of non-Archimedean
numbers from classical analysis and the non-symmetry of analytic continuation
for GHF. The introduction of hyperseries allows one to recover all
these features including also interesting non trivial generalized
functions and compactly supported functions. This development paves
the way for an interesting generalization of the Cauchy-Kowalevski
theorem for GHF that we intend to study in the next work of this series.

\section*{Acknowledgements}

Financial support for an Ernst-Mach PhD grant (Reference number =
MPC-2021-00491) issued by the Federal Ministry of Education, Science
and Research (BMBWF) through the Austrian Agency for Education and
Internationalization (OeAD-Gmbh) within the framework ASEA-UNINET
(https://asea-uninet.org/) for S. Nugraheni is gratefully acknowledged.
This research was also funded in whole or in part by the Austrian
Science Fund (FWF) P34113, 10.55776/P33945, 10.55776/P33538. For open
access purposes, the author has applied a CC BY public copyright license
to any author-accepted manuscript version arising from this submission.

\end{document}